\documentclass[11pt,a4paper,oneside,centertags,reqno]{amsart}

\usepackage{amsmath}
\usepackage{amsfonts}
\usepackage{amssymb}
\usepackage{footmisc}
\usepackage{enumerate}
\usepackage{hyperref}

\font\got=eufm10 at 12pt

\newcommand{\La}{\Lambda}

\newcommand{\tPi}{Q_t^i}
\newcommand{\tP}{Q}
\newcommand{\la}[0]{{\lambda}}
\newcommand{\lki}[0]{\lambda^{i}_{k_i}}
\newcommand{\lk}[0]{{\lambda_k}}
\newcommand{\mD}[0]{{\mathcal D}}

\renewcommand{\d}[0]{{\text{\got d}}}

\newcommand{\f}[0]{\varphi}
\newcommand{\fki}[0]{\varphi^i_{k_i}}
\newcommand{\fk}[0]{\varphi_{k}}
\newcommand{\g}[0]{{\bf g}}

\newcommand{\tL}{\tilde{L}}

\newcommand{\sk}[2]{\left\langle #1 , #2\right\rangle}

\renewcommand{\geq}[0]{\geqslant}
\renewcommand{\leq}[0]{\leqslant}

\newcommand{\N}[0]{{\mathbb N}}

\newcommand{\R}[0]{\mathbb{R}}

\DeclareMathOperator{\lin}{lin}
\DeclareMathOperator{\Dom}{Dom}
\DeclareMathOperator{\Hess}{Hess}

\newcommand{\Ld}{\tilde{L}}
\newcommand{\pd}[0]{\partial}

\renewcommand{\theta}[0]{\vartheta}

\newtheorem{thm}{Theorem}
\newtheorem*{thms}{Main result (informal)}
\newtheorem{lema}[thm]{Lemma}
\newtheorem{prop}[thm]{Proposition}
\newtheorem*{prop*}{Proposition}

\theoremstyle{remark}
\newtheorem*{remark}{Remark}
\newtheorem*{remark1}{Remark 1}
\newtheorem*{remark2}{Remark 2}

\begin{document}

\title[Dimension-free  $L^p$ estimates for vectors of Riesz transforms]{Dimension-free $L^p$ estimates for vectors of Riesz transforms associated with orthogonal expansions}

\author{B\l a\.{z}ej Wr\'obel}
\address{
	Mathematical Institute
	\\ 	Universit\"at Bonn\\
	Endenicher Allee 60\\ D--53115 Bonn\\ Germany
	\newline \&
	Instytut Matematyczny, Uniwersytet Wroc\l awski,
	pl. Grunwaldzki 2/4, 50-384 Wroc\l aw, Poland}
\email{blazej.wrobel@math.uni.wroc.pl}

\begin{abstract}
	An explicit Bellman function is used to prove a bilinear embedding theorem for operators associated with general multi-dimensional orthogonal expansions on product spaces. This is then applied to obtain $L^p,$ $1<p<\infty,$ boundedness of appropriate vectorial Riesz transforms, in particular in the case of Jacobi polynomials. Our estimates for the $L^p$ norms of these Riesz transforms are both dimension-free and linear in $\max(p,p/(p-1)).$ The approach we present allows us to avoid the use of both differential forms and general spectral multipliers. 
\end{abstract}

\maketitle

 \numberwithin{equation}{section}
\section{Introduction}

The classical Riesz transforms on $\R^d$ are the operators
$$R_if(x)=\partial_{x_i}(-\Delta_{\R^d})^{-1/2}f(x),\qquad i=1,\ldots,d.$$
In \cite{St1} E. M. Stein proved that the vector of Riesz transforms
$${\bf R }f=(R_1f,\ldots,R_df)$$
 has $L^p$  bounds which are independent of the dimension. More precisely
 \begin{equation}
 \label{eq:Riesz0}
 \|{\bf R}f\|_{L^p(\R^d)}\leq C_p\, \|f\|_{L^p(\R^d)},\qquad 1<p<\infty,\end{equation}
 where $C_p$ is independent of the dimension $d.$ Note that \eqref{eq:Riesz0} is formally the same as the a priori bound
 \begin{equation*}
\big\||\nabla f|\big\|_{L^p(\R^d)}\leq C_p\, \big\|(-\Delta)^{1/2}f\big\|_{L^p(\R^d)}.\end{equation*} 
 Later it was realized that, for $1<p<2$ one may take $C_p\leq C (p-1)^{-1}$ in \eqref{eq:Riesz0}, see \cite{Ba1}, \cite{RubDuo1}. It is worth mentioning that the best constant in \eqref{eq:Riesz0} remains unknown when $d\ge 2;$ the best results to date are given in \cite{BW1} (see also \cite{DV} for an analytic proof) and \cite{IM1}.  

The main goal of this paper is to generalize \eqref{eq:Riesz0} to product settings different from $\R^d=\R\times \cdots \times \R$ with the product Lebesgue measure. Our starting point is the observation that the classical Riesz transform can be written as
$R_i =\delta_i(\sum_{i=1}^d L_i)^{-1/2}$
where 
$\delta_i=\partial_{x_i},$   and $L_i=\delta_i^*\delta_i.$
The generalized Riesz transforms we pursue are of the same form
\begin{equation}
\label{eq:Rieszgen0}
R_i =\delta_iL^{-1/2},\qquad i=1,\ldots,d,
\end{equation} 
with $\delta_i$ being an operator on $L^2(X_i,\mu_i)$,  $$L_i=\delta_i^*\delta_i+a_i,\qquad \textrm{and}\qquad  L=\sum_{i=1}^d L_i.$$ 
Here $a_i$ is a non-negative constant. The adjoint  $\delta_i^*$ is taken with respect to the inner product on $L^2(X_i,\mu_i),$ where $\mu_i$ is a non-negative Borel measure on $X_i$ such that $d\mu_i(x_i)=w_i(x_i)dx_i$ for some positive and smooth function $w_i$ on $X_i.$ To be precise, if $0$ is an $L^2$ eigenvalue of $L,$ then the definition of $R_i$ needs to be slightly modified; this is properly explained in the next section. Throughout the paper we assume that each $X_i,$ $i=1,\ldots,d,$ is an open interval in $\R$, an open half-line in $\R$ or is the real line; we also set $X=X_1\times \cdots \times X_d$ and $\mu=\mu_1\otimes \cdots \otimes \mu_d.$  We consider $\delta_i$ being given by
$$\delta_i f(x)=p_i(x_i)\,\partial_{x_i}+q_i(x_i),\qquad x_i \in X_i,$$
for some real-valued functions $p_i \in C^{\infty}(X_i)$ and $q_i\in C^{\infty}(X_i).$ We remark that a significant difference between the classical Riesz transforms and the general Riesz transforms \eqref{eq:Rieszgen0} lies in the fact that the operators $\delta_i$ and $\delta_i^*$ do not need to commute. 

There are two assumptions which are critical to our results. Firstly, a computation, see \cite[p. 683]{NS2}, shows that the commutator $[\delta_i,\delta_i^*]$ is a function which we call $v_i$. We assume that $v_i$ is non-negative, cf.\ \eqref{eq:A1}. Secondly, it is not hard to see that $L=\sum_{i=1}^d L_i$ may be written as $L=\tL+r,$ where $\tL$ is a purely differential operator (without a zero order potential term) and $r$ is the potential term. We impose that $\sum_{i=1}^d q_i^2$  is controlled pointwise from above by a constant times $r,$ namely $\sum_{i=1}^d q_i^2\le K\cdot r,$ for some $K\geq 0,$  cf.\ \eqref{eq:A2}. In several cases we will consider we can take $K=1$ or $K=0.$ In particular if $q_1=\cdots=q_d=0$ then the bound \eqref{eq:A2} holds with $K=0.$  When $0$ is not an $L^2$ eigenvalue of $L$ our main result can be summarized as follows.    
\begin{thms}
Set  $p^*=\max(p,p/(p-1)).$ Then the vectorial Riesz transform ${\bf R} f=(R_1f,\ldots,R_df)$ with $R_i$ given by  \eqref{eq:Rieszgen0} satisfies the bounds
$$\left\|{\bf R }f\right\|_{L^p(X,\mu)}\leq 24(1+\sqrt{K})(p^*-1)\|f\|_{L^p(X,\mu)},\qquad 1<p<\infty.$$
In other words, introducing $\delta f=(\delta_1 f,\ldots,\delta_d f),$ we have
$$\big\||\delta f|\big\|_{L^p(X,\mu)}\leq 24(1+\sqrt{K})(p^*-1)\left\|L^{1/2}f\right\|_{L^p(X,\mu)},\qquad 1<p<\infty.$$
\end{thms}  

The rigorous statement of our main result is contained in Theorem \ref{thm:main}. In order to prove it we need some extra technical assumptions. For the sake of clarity of the presentation we decided to concentrate on the case of orthogonal expansions, when each of the operators $L_i=\delta_i^*\delta_i+a_i$ has a decomposition in terms of an orthonormal basis. Our precise setting is described in detail in Section \ref{Sec:Prelim}. We follow the approach of Nowak and Stempak from \cite{NS2}, in fact the present paper may be thought of as an $L^p$ counterpart for a large part of the $L^2$ results from \cite{NS2}. Adding the technical assumptions \eqref{eq:T1}, \eqref{eq:T2}, and  \eqref{eq:T3} to the crucial assumptions \eqref{eq:A1} and \eqref{eq:A2} we state our main result Theorem \ref{thm:main} in Section \ref{sec:GenRiesz}. In all the cases we will consider, the projection $\Pi$ appearing in Theorem \ref{thm:main} is the identity operator or has its $L^p$ norm bounded by $2$ for all $1\le p\le \infty.$ Moreover, we have $\Pi=I$ if and only if $0$ is not an $L^2$ eigenvalue of $L.$

 From Theorem \ref{thm:main} we obtain several new  dimension-free bounds on $L^p,$ $1<p<\infty,$ for vectors of Riesz transforms connected with classical multi-dimensional orthogonal expansions. For more details we refer to the examples in Section \ref{sec:Exa}. For instance in Section \ref{sec:Jac} we obtain the dimension-free boundedness for the vector of Riesz transforms in the case of Jacobi polynomial expansions. This answers a question left open in Nowak and Sj\"ogren's \cite{NSj2}. Moreover, the approach we present gives a unified way to treat dimension-free estimates for vectors of Riesz transforms. In most of the previous cases separate papers were written for each of the classical orthogonal expansions. More unified approaches were recently presented by Forzani, Sasso, and Scotto in \cite{FSS1} and by the author in \cite{BWr1}. However, these papers treat only dimension-free estimates for scalar Riesz transforms and not for the vector of Riesz transforms.

Let us remark that Theorem \ref{thm:main} formally cannot be applied to some cases where the crucial assumptions on $v_i$ and $r$ continue to  hold. This is true when $L$ has a purely continuous spectrum,  for instance for the classical Riesz transforms on $\R^d$ (when $v_i=0$ and $r=0$). However, it is not difficult to modify the proof of Theorem \ref{thm:main} so that it remains valid for the classical Riesz transforms. We believe that a similar procedure can be applied to other cases outside the scope of Theorem \ref{thm:main}, as long as the crucial assumptions \eqref{eq:A1} and \eqref{eq:A2} are satisfied. 

We deduce Theorem \ref{thm:main} from a bilinear embedding theorem  (see Theorem \ref{thm:bilem}) together with a bilinear formula (see Proposition \ref{pro:Form1}). The main tool that is used to  prove Theorem \ref{thm:bilem} is the Bellman function technique. This method was introduced to harmonic analysis by Nazarov, Treil, and Volberg \cite{NTV1}. Before \cite{NTV1} Bellman functions appeared implicitly in the work of Burkholder \cite{Bur1}, \cite{Bur2}, \cite{Bur3}. The proof of Theorem \ref{thm:bilem} is presented in Section \ref{sec:BilEm} and is based on subtle properties of a particular Bellman function. This approach was devised by Dragi\v{c}evi\'c and Volberg in \cite{DV,DV-Kato,DV-Sch}.
Carbonaro and Dragi\v{c}evi\'c developed the method further in \cite{CD}, \cite{CD-mult}, \cite{CD-NonSymOu}, and \cite{CD-DivCom}. The approach from \cite{CD} was recently adapted by Mauceri and Spinelli in \cite{MSJFA} to the case of the Laguerre operator. Our paper generalizes simultaneously \cite{DV-Sch} (as we admit a non-negative potential $r$) and \cite{DV}, \cite{MSJFA} (as we consider general $p_i$ in $\delta_i=p_i\partial_{x_i}+q_i$).    

In some applications of the Bellman function method the authors needed to prove dimension-free bounds on $L^p$ for certain spectral multipliers related to the considered operators, see \cite{DV} and \cite{DV-Sch} for such a situation. In other papers mentioned in the previous paragraph they needed to consider operators acting on differential forms, cf. \cite{CD} and \cite{MSJFA}. One of the merits of our approach is that we avoid to use both general spectral multipliers and differential forms. This is achieved by means of the bilinear formula from Proposition \ref{pro:Form1}. This formula relates the Riesz transform $R_i$ with an integral where only $\delta_i$ and two kinds of semigroups (one for $L$ and one for $L+v_i$)  are present,  see \eqref{visby}. 

For the sake of simplicity we use a Bellman function with real entries in Section \ref{sec:BilEm}. Thus our main results Theorems \ref{thm:main} and \ref{thm:bilem} apply to real-valued functions. Of course they can be easily extended to complex valued-functions with the constants being twice as large. One may improve the estimates further by using a Bellman function with complex arguments as it was done in \cite{DV}, \cite{DV-Kato},  and \cite{DV-Sch}.

{\bf Notations.} We finish this section by introducing the general notations used in the paper. By $\N$ we denote the set of non-negative integers. For $N\in \N$ and $Y$ being an open subset of $\R^N$ the symbol $C^{n}(Y),$ $n\in \N,$ denotes the space of real-valued functions which have continuous partial derivatives in $Y$ up to the order $n.$ In particular $C^0(Y)=C(Y)$ denotes the space of continuous functions on $Y$ equipped with the supremum norm. By $C^{\infty}(Y)$ we mean the space of infinitely differentiable functions on $Y.$ Whenever we say that $\nu$ is a measure on $Y$ we mean that $\nu$ is a Borel measure on $Y$. The symbols $\nabla f$ and $\Hess f$ stand for the gradient and the Hessian of a function $f\colon \R^N\to \R.$ For  $a,b\in \R^N,$ we denote by $\sk{a}{b}$ the inner product on $\R^N$ and set $|a|^2=\sk{a}{a}.$ The actual $N$ should be clear from the context (in fact we always have $N\in\{1,d,d+1\}$). For $p\in (1,\infty)$ we set
$$p^*=\max\bigg(p,\frac{p}{p-1}\bigg).$$

\section{Preliminaries}
\label{Sec:Prelim}
 All the functions we consider are real-valued. Our notations will closely follow that of \cite{NS2}.  

For $i=1,\ldots,d,$ let $X_i$ be the real line $\R,$ an open half-line in $\R$ or an open interval in $\R$  of the form
$$X_i=(\sigma_i,\Sigma_i),\textrm{ where }-\infty\leq \sigma_i<\Sigma_i\leq \infty.$$ Consider the measure spaces $(X_i,\mathcal{B}_i,\mu_i),$ where $\mathcal{B}_i$ denotes the $\sigma$-algebra of Borel subsets of $X_i$ and $\mu_i$ is a Borel measure on $X_i.$ 
We impose that $d\mu_i(x_i)=w_i(x_i)\,dx_i,$ where $w_i$ is a positive $C^{\infty}$ function on $X_i.$ Note that in \cite{NS2} the authors assumed that $X_1=\cdots=X_d;$ this is however not needed in our paper. Throughout the article we let $$X=X_1\times\cdots \times X_d,\qquad \mu=\mu_1\otimes \cdots \otimes \mu_d,$$ and abbreviate $$L^p:=L^p(X,\mu),\qquad \|\cdot\|_{p}=\|\cdot\|_{L^p},\qquad \textrm{and}\qquad  \|\cdot\|_{p\to p}=\|\cdot \|_{L^p\to L^p}.$$ This notation is also used for vector-valued functions. Namely, if ${\bf g}=(g_1,\ldots,g_N)\colon X\to \R^N,$ for some $N\in \N,$ then 
$$\|{\bf g}\|_p=\bigg(\int_{X}|{\bf g}(x)|^p\,d\mu(x)\bigg)^{1/p},\qquad\textrm{with}\qquad |{\bf g}(x)|=\bigg(\sum_{i=1}^N|g_i(x)|^2\bigg)^{1/2}.$$
We shall also write $\sk{f}{g}_{L^2}$ for $\sk{f}{g}_{L^2(X,\mu)}.$ 

Let $\delta_i,$ $i=1,\hdots,d,$ be the operators acting on $C_c^{\infty}(X_i)$ functions via
$$
\delta_i=p_i\pd_{x_i}+q_i.
$$
Here $p_i$ and $q_i$ are real-valued functions on $X_i$,  with $p_i\in C^{\infty}(X_i)$ and $q_i\in C^{\infty}(X_i).$ We assume that $p_i(x_i)\neq 0,$ for $x_i \in X_i.$ We shall also denote by $p$ and $q$ the exponents of $L^p$ and $L^q$ spaces. This will not lead to any confusion as the functions $p_i$ and $q_i$ will always appear with the index $i=1,\ldots,d$. 

Let  $\delta_i^*$ be the formal adjoint of $\delta_i$ with respect to the inner product on $L^2(X_i,\mu_i),$ i.e.
$$\delta_i^*f=-\frac1{w_i}\partial_{x_i}\big(p_i w_i f\big)+q_if,\qquad f\in C_c^{\infty}(X_i).$$
 A simple calculation, see \cite[p. 683]{NS2}, shows that the commutator \begin{equation}
\label{eq:com1}
[\delta_i,\delta_i^*]=\delta_i\delta_i^*-\delta_i^*\delta_i=p_i\left(2q_i'-\left(p_i\frac{w_i'}{w_i}\right)'-p_i''\right)=:v_i
\end{equation} is a locally integrable function (0-order operator). 
Most of the assumptions made in this section are of a technical nature. The first of the two assumptions that are crucial to our results is the following: 
\begin{equation}
\tag{A1}
\label{eq:A1}
\textit{the functions }v_i, i=1,\ldots,d,\textit{ are non-negative}.
	\end{equation}
The property \eqref{eq:A1} has been (explicitly or implicitly) instrumental  for establishing the main results in \cite{HRST}, \cite{MSJFA}, \cite{NSj2}, \cite{StWr1}. It is also explicitly stated by Forzani, Sasso, and Scotto as Assumption H1 c) in \cite{FSS1}.

For a scalar $a_i\ge 0$ we let $L_i$ and $L$ to be given on $C_c^{\infty}(X)$ by 
$$
L_i:=\delta_i^*\delta_i+a_i,\qquad L=\sum_{i=1}^d L_i.
$$
Here each $L_i$ can be considered to act either on $C_c^{\infty}(X_i)$ or on $C_c^{\infty}(X),$ thus the definition of $L$ makes sense. Note that both $L_i$ and $L$ are symmetric on $C_c^{\infty}(X)$ with respect to the inner product on $L^2.$ We assume that for each $i=1,\ldots,d,$ there is an orthonormal basis $\{\fki\}_{k_i\in \N}$ which consists of $L^2(X_i,\mu_i)$ eigenvectors of $L_i$ that correspond to non-negative eigenvalues $\{\lambda_{k_i}^i\}_{k_i\in \N},$ i.e.
$$L_i \fki=\lki \fki.$$ 
Then, it must be that $\la_{k_i}\ge a_i,$ for $k_i\in \N$ and $i=1,\ldots,d.$  
We require that the sequence $\{\lki\}_{k_i\in \N}$ is strictly increasing and that $\lim_{k_i\to \infty}\lki=\infty.$  
Note that our assumptions on $p_i,q_i,$ and $w_i$ imply that $L_i$ is hypoelliptic. Therefore  we have $\fki\in C^{\infty}(X_i).$   
Setting, for $k=(k_1,\ldots,k_d)\in \N^d,$ 
\begin{equation}
\label{eq:fkdef}
\fk=\f^{1}_{k_1}\otimes \cdots \otimes \f^{d}_{k_d}, \end{equation}
we obtain an orthonormal basis of eigenvectors on $L^2$ for the operator $L=L_1+\cdots + L_d.$  The eigenvalue corresponding to $\fk$ is $$\lk:=\la^{1}_{k_1}+\cdots+ \la^d_{k_d},$$
so that $L \fk= \lk \fk.$  
We consider the self-adjoint extension of $L$ (still denoted by the same symbol) given by
$$Lf=\sum_{k\in\N^d}\la_k\,\sk{ f}{\f_{k}}_{L^2}\f_{k}$$
on the domain
$$\Dom(L)=\{f\in L^2\colon \sum_{k\in\N^d}|\la_k|^2|
\sk{ f}{\f_{k}}_{L^2}|^2<\infty\}.$$ 
We assume that the eigenfunctions $\fki,$ $i=1,\ldots,d,$ are such that
\begin{equation}
\tag{T1}
\label{eq:T1}
\sk{\delta_i \fki}{\delta_i \f^{i}_{m_i}}_{L^2(X_i,\mu_i)}=\sk{\delta_i^*\delta_i\, \fki}{ \f^{i}_{m_i}}_{L^2(X_i,\mu_i)},
\end{equation}
for $i=1,\ldots,d,$ and $k_i,m_i\in \N,$ cf.\ \cite[eq.\ (2.8)]{NS2}. The condition \eqref{eq:T1} implies that the functions \begin{equation}
\label{eq:dfkidef}\delta_i \fk=\f^1_{k_1}\otimes\cdots \otimes  \delta_i \f_{k_i}^{i}\otimes \cdots\otimes \f_{k_d}^d\end{equation}
are pairwise orthogonal on $L^2$ and
$$\sk{\delta_i \fk}{\delta_i \f_{k}}_{L^2}=\la_{k_i}^{i}-a_i.$$
cf.\  \cite[Lemma 5,6]{NS2}. Moreover, since $\fk\in C^{\infty}(X)$ we also see that $\delta_i \fk \in C^{\infty}(X).$  

We remark that our assumptions differ slightly from those in 
\cite{NS2}. Namely, we assume that the coefficients $p_i,q_i,$ and the 
weight $w_i$ are  $C^{\infty}$ functions, whereas in \cite{NS2} the 
authors considered $p_i,q_i,w_i$ that possessed only a finite 
order of smoothness. The smoothness of these functions is in fact needed 
to easily conclude that $L_i$ is hypoelliptic and that $\fk \in 
C^{\infty}(X)$, which is an issue that was overlooked\footnote[2]{The hypoellipticity of $L_i$ is not necessary for the theory from 
	\cite{NS2} to work \cite{Npriv}. When not having this property one has to 
	add instead some extra assumptions (much weaker than smoothness) on 
	the regularity of the eigenfunctions $\fk$.} in 
\cite{NS2}.

We also impose a boundary condition on the functions $\fki$ and $\delta_i \fki.$ 
Namely, we require that for each $i=1,\ldots,d,$ if  $z_i\in\{\sigma_i,\Sigma_i\},$ then
\begin{equation}
\tag{T2}
\label{eq:T2}
\begin{split}
\lim_{x_i\to z_i}\left[(1+|\fki|^{s_1}+|\delta_i \fki|^{s_2})(p_i^2\, w_i\, \pd_{x_i} \fki)\right](x_i)&=0,\\
\lim_{x_i\to z_i}\left[(1+|\fki|^{s_1}+|\delta_i \fki|^{s_2})(p_i^2\, w_i\, \pd_{x_i} \delta_i \fki)\right](x_i)&=0,
\end{split}
\end{equation} 
for all $k_i\in \N $ and $s_1,s_2> 0.$  
Condition \eqref{eq:T2} is close to the assumption H1 a) from \cite{FSS1}. Observe that the term $|\fki|^{s_1}+|\delta_i \fki|^{s_2}$ in \eqref{eq:T2} is significant only when the functions $\fki$ and $\delta_i \fki$ are unbounded on $X_i.$

Let 
$$A=a_1+\cdots+a_d,\qquad \Lambda_0=\la_0^1+\cdots +\la_0^d.$$
Then $\Lambda_0 $ is the smallest eigenvalue of $L.$ We set
$$\N^d_{\La}=\left\{\begin{tabular}{ c  }
$\N^d,\hspace{2.5cm}  \Lambda_0>0$\\
$\N^d\setminus\{(0,\ldots,0)\},\quad \Lambda_0=0.$
\end{tabular}\right.$$
and define 
$$\Pi f=\sum_{k\in \N^d_{\Lambda}}\sk{f}{\f_{k}}_{L^2}\f_{k}.$$
Then in the case $\Lambda_0>0$ we have $\Pi=I,$ while in the case $\Lambda_0=0$ the operator $\Pi$ is the projection onto the orthogonal complement of the vector $\f_{(0,\ldots,0)}.$ 
The Riesz transforms studied in this paper are formally of the form
$$R_i:=\delta_i L^{-1/2}\Pi,$$
while the rigorous definition of $R_i$ is 
\begin{equation*}
R_i f= \sum_{k\in \N^d_{\Lambda}}\la_{k}^{-1/2}\sk{ f}{\f_{k}}_{L^2}\delta_i\f_{k}.
\end{equation*} 
In many of the considered cases  $\Pi\equiv I$ so that $R_i=\delta_iL^{-1/2}.$

It was proved in \cite[Proposition 1]{NS2} that the vector of Riesz transforms $${\bf R}f=(R_1 f,\ldots,R_d f)$$ satisfies
$$\|{\bf R} f\|_{2\to 2}\leq \|f\|_2.$$
The main goal of this paper is to prove similar estimates for $p$ in place of $2.$ We aim at these estimates being dimension-free and linear in $p^*.$ More precisely, we shall prove that for $1<p<\infty$ it holds 
 $$\|{\bf R} f\|_{p\to p}\leq C(p^*-1)\|f\|_p.$$
 Here $C$ is a constant that is independent of both $p$ and the dimension $d.$

To state and prove our main results we need several auxiliary objects. Firstly, we let 
\begin{equation}
\label{eq:didef}\d_i=p_i\partial_{x_i}.\end{equation}
That is, $\d_i$ is the 'differential' part of $\delta_i.$  In many (though not all) of our applications we will have $q_i\equiv 0$ and thus $\delta_i\equiv \d_i.$ The formal adjoint of $\d_i$ on $L^2(X_i,\mu_i)$ is
\begin{equation}
\label{eq:dderadj}
\d_i^*f=-\frac1{w_i}\partial_{x_i}\big(p_i w_i f\big),\qquad f\in C_c^{\infty}(X_i).
\end{equation} A computation shows that
 $L_i=\d_i^*\d_i+r_i,$
 with \begin{equation}
 \label{eq:rform1}
 r_i=a_i +\bigg(q_i^2-p_iq_i'-p_i'q_i -p_iq_i \frac{w_i'}{w_i}\bigg).\end{equation} We shall also need
 $$\Ld:=\sum_{i=1}^d \d_i^*\d_i=L-r,\qquad \textrm{where }r:=\sum_{i=1}^d r_i.$$
Then $\Ld$ is the potential-free component of $L$ and the potential $r$ is a locally integrable function on $X$. We assume that  
 \begin{equation}
 \tag{A2}
 \label{eq:A2}
\textit{there is a constant $K\geq 0$ such that } \sum_{i=1}^d q_i^2(x_i)\le  K\cdot  r(x),
 \end{equation} 
 for all $x\in X.$ This is our second (and last) crucial assumption.
In many of our examples we shall have $q_1=\cdots=q_d=0$ and thus $r=A$ and \eqref{eq:A2} holding with $K=0.$

Next we define
 $$
 M_i:=\sum_{j\neq i}\delta_j^*\delta_j+\delta_i\delta_i^*=L+[\delta_i,\delta_i^*]=L+v_i,
 $$
 cf.\ \cite[(eq. 5.1)]{NS2},  and set $$c^i_{k}=\|\delta_i \fk \|_2^{-1},$$ if $\delta_i \fk\ne 0$ and $c_k^i=0$ in the other case. 
Then $\{c_k^i\delta_i \fk\}_{k\in \N^d}$ (excluding those of $c_k^i\delta_i \fk$ which vanish) is an orthonormal system of eigenvectors of $M_i$ such that $M_i(c_k^i\delta_i \fk)$ equals $ \la_k c_k^i\delta_i \fk.$ 

We denote
$$\mD=\lin\{\fk\colon k\in \N^d\},\qquad \mD_i=\delta_i[\mD]=\lin\{\delta_i \fk\colon k\in \N^d\},$$
and make the technical assumption that
\begin{equation}
\tag{T3}
\label{eq:T3}
\textrm{both $\mD$ and $\mD_i,$ $i=1,\ldots,d$, are dense subspaces of $L^p$, $1\le p<\infty$}.
\end{equation}
In most of our applications the condition \eqref{eq:T3} will follow from \cite[Lemma 7.5]{FSS1}, which is itself a consequence of \cite[Theorem 5]{BC1}.
\begin{lema}[{\cite[Lemma 7.5]{FSS1}}]
	\label{lem:BC1}
	Assume that $\nu$ is a measure on $X$ such that, for some $\varepsilon>0$ we have
	$$\int_{X}\exp\bigg({\varepsilon\sum_{i=1}^d|y_i|}\bigg)\,d\nu(y)<\infty.$$  
	Then, for each $1\leq p<\infty,$ multivariable polynomials on $X$ are dense in $L^p(X,\nu).$
\end{lema}
%
%

 In what follows we consider the self-adjoint extension of $M_i$ given by
\begin{equation*}
M_if=\sum_{k\in \N^d}\la_k \sk{ f}{c_k^i\,\delta_i\f_{k}}_{L^2}c_k^i\,\delta_i\f_{k},\end{equation*}
on the domain
$$\Dom(M_i)=\{f\in L^2\colon \sum_{k\in\N^d}|\la_k|^2|\sk{ f}{c_k^i\delta_i\f_{k}}_{L^2}|^2<\infty\}.$$
Keeping the symbol $M_i$ for this self-adjoint extension is a slight abuse of notation, which however will not lead to any confusion.
Finally, we shall need the semigroups 
 $$P_t:=e^{-tL^{1/2}}\qquad \textrm{and}\qquad \tP_t^{i}:=e^{-tM_i^{1/2}}.$$
 These are formally defined on $L^2$ as
 $$P_tf= \sum_{k\in \N^d}e^{-t\la_{k}^{1/2}}\sk{ f}{\f_{k}}_{L^2}\f_{k},\qquad  \tPi f=\sum_{k\in \N^d}e^{-t\la_k^{1/2}}\,\sk{ f}{c_k^i\,\delta_i\f_{k}}_{L^2}c_k^i\,\delta_i\f_{k}.$$
 Note that for $t>0$ we have $P_t[\mD]\subseteq \mD$ and $Q_t^i[\mD_i]\subseteq \mD_i,$ $i=1,\ldots,d.$  

\section{General results for Riesz transforms}
\label{sec:GenRiesz}
Recall that we are in the setting of the previous section. In particular the assumptions \eqref{eq:A1}, \eqref{eq:A2}, and the technical assumptions \eqref{eq:T1}, \eqref{eq:T2}, \eqref{eq:T3}, are in force. The following is  the main result of our paper. 
\begin{thm}
	\label{thm:main}
For each $1<p<\infty$ we have
	\begin{equation*}
	\left\|{\bf R} f\right\|_p\leq 24(1+\sqrt{K})(p^*-1)\|\Pi f\|_{L^p},\qquad f\in L^p.
	\end{equation*}
\end{thm}
\begin{remark}
	In all the examples we consider in Section \ref{sec:Exa} the projection $\Pi $ satisfies $\|\Pi \|_{p\to p}\leq 2,$ $1\leq p\leq \infty.$ In fact in many of the examples $\Pi$ equals the identity operator. 
\end{remark}
In order to prove Theorem \ref{thm:main} we need two ingredients. The first of these ingredients is a bilinear formula that relates the Riesz transform with an integral in which both $P_t$ and $\tPi$ are present. 
\begin{prop}
	\label{pro:Form1}
	Let $i=1,\ldots,d.$ Then the formula
	\begin{equation}
	\label{visby}
	\sk{R_i f}{g}_{L^2}=-4\int_0^\infty\sk{\delta_i P_t\Pi f}{\pd_t \tPi g}_{L^2}\,t\,dt,
	\end{equation}
	holds for $f\in \mD $ and $g\in \mD_i.$ 
\end{prop}
Before proving the proposition let us make two remarks.
\begin{remark1}
	Formulas similar to \eqref{visby} were proved before, though, depending on the context, they may have involved spectral multipliers of the operator $L$. However, treating these spectral multipliers appropriately was achieved with variable success. A way of avoiding multipliers was first devised in \cite{CD} for Riesz transforms on manifolds. In such a setting, the above formula is a special case of the identity (3) there.
	The approach in \cite{CD} was adapted in \cite{MSJFA} to the case of Hodge-Laguerre operators. In the case of Laguerre polynomial expansions (see Section \ref{LagPol}) the formula \eqref{visby} is a special case of  \cite[eq.\ (5.1)]{MSJFA}. We note that both in \cite{CD} and \cite{MSJFA} the authors needed to consider the Riesz transform as well as the formula \eqref{visby} for differential forms; this is not needed in our approach.
\end{remark1}
\begin{remark2}
	Note that if the operators $\delta_i$ and $\delta_i^*$ commute, then $\tPi=P_t$ and the formula \eqref{visby} can be formally obtained via the spectral theorem. The problem is that often these operators do not commute. A way to overcome this non-commutativity problem was devised by
	Nowak and Stempak in \cite{NS}. They introduced a symmetrization $T_i$ of $\delta_i$ that does commute with its adjoint, in fact $T_i^*=-T_i.$ These symmetrization is defined on $L^2(\tilde{X}),$ where $$\tilde{X}=(X_1\cup(-X_1))\times\cdots\times (X_d\cup (-X_d)).$$ Set $T=-\sum_{i=1}^d T_i^2$ and let $S_t=e^{-tT^{1/2}}.$ The formula \eqref{visby} for $T_i$ is then formally
	\begin{equation}
	\label{eq:symform}
	\sk{T_i T^{-1/2}f}{g}_{L^2(\tilde{X})}=-4\int_0^\infty\sk{T_i S_tf}{\pd_t S_tg}_{L^2(\tilde{X})}\,t\,dt.\end{equation}
	This leads to a proof of \eqref{visby} different from the one presented in our paper. Namely, a computation shows that applying \eqref{eq:symform} to functions $f\colon \tilde{X}\to \R$ and $g\colon \tilde{X}\to \R$ which are both even in all the variables we arrive at \eqref{visby}.
\end{remark2}
\begin{proof}[Proof of Proposition \ref{pro:Form1}]	
	
	We start with proving \eqref{visby} for $f=\fk$ and $g=\delta_i \f_n,$ with some $k,n\in \N^d.$ If $k=0$ and $\La_0=0$ then both sides of \eqref{visby} vanish. Thus we can assume that $\lambda_k>0.$ A computation 
	shows that
	$$
	\sk{\delta_i L^{-1/2}f}{g}_{L^2}=
	\lambda_k^{-1/2}\sk{\delta_if}{g}_{L^2}
	$$
	and
	$$
	\aligned
	-4\int_0^\infty\sk{\delta_i P_tf}{\pd_t Q^{i}_tg}_{L^2}\,t\,dt&=-4\int_0^\infty\sk{e^{-t\lambda_k^{1/2}}\delta_if}{-\la_n^{1/2}e^{-t\la_n^{1/2}}g}_{L^2}t\,dt\\
	&=4\la_n^{1/2}\int_0^\infty e^{-t(\lambda_k^{1/2}+\la_n^{1/2})}t\,dt\cdot\sk{\delta_if}{g}_{L^2}\\
	&=\frac{4\la_n^{1/2}}{(\lambda_k^{1/2}+\la_n^{1/2})^2}\cdot\sk{\delta_if}{g}_{L^2},
	\endaligned
	$$
	hence
	\begin{equation}
	\label{eq:1}
	\begin{split}
	\sk{\delta_i L^{-1/2}f}{g}+4\int_0^\infty\sk{\delta_i P_tf}{\pd_t Q^{i}_tg}_{L^2}\,t\,dt\\=
	\left(\lambda_k^{-1/2}-\frac{4\la_n^{1/2}}{(\lambda_k^{1/2}+\la_n^{1/2})^2}\right)
	\cdot\sk{\delta_if}{g}_{L^2}.
	\end{split}
	\end{equation}
	Now $\delta_i f$ is also an $L^2$ eigenvector for $M_i$ corresponding to the eigenvalue $\lambda_k$. Consequently, since eigenspaces for $M_i$ corresponding to different eigenvalues are orthogonal, $\sk{\delta_if}{g}$ is nonzero only if $\la_n=\la_k.$ Coming back to \eqref{eq:1} we obtain \eqref{visby} for $f=\fk$ and $g=\delta_i \f_n.$ 
	
	Finally, by linearity \eqref{visby} holds also for $f\in \mD$ and $g\in \mD_i.$ 
\end{proof}

The second ingredient we need to prove Theorem \ref{thm:main} is a  bilinear embedding, as was the case in \cite{CD,DV,DV-Sch,MS}. For $N\in\N$ (the cases interesting to us being $N=1$ and $N=d$) we take $F=(f_1,\hdots,f_N):X\times (0,\infty)\rightarrow\R^N$ and set
\begin{equation}
\label{john coltrane}
|F|_*^2
:=r|F|^2+|\pd_t F|^2+\sum_{i=1}^d|\d_i F|^2.
\end{equation}
The absolute values $|\cdot|$ in \eqref{john coltrane} denote the Euclidean norms on $\R^N$ of the vectors $F(x,t),$ $\pd_t F(x,t)=(\pd_t f_1(x,t),\ldots, \pd_t f_N(x,t)),$ and $\d_i F(x,t)=(\d_i f_1(x,t),\ldots, \d_i f_N(x,t)),$ where $(x,t)\in X\times(0,\infty).$    
Below we only state our bilinear embedding. The proof of it is presented in the next section.
\begin{thm}
	\label{thm:bilem}
Let $f:X\rightarrow\R$ and $\g=(g_1,\hdots,g_d):X^d\rightarrow\R^d$ and assume that $f\in \mD$ and $g_i\in \mD_i,$ for $i=1,\ldots,d.$   Denote 
$$
F(x,t)=P_t\,\Pi f(x)\qquad \textrm{and}\qquad G(x,t)=Q_t{\bf g}=\big(\tP_t^{1}g_1,\hdots,\tP_t^{d}g_d\big).$$
Then
\begin{equation}
\label{eq:18}
\int_0^\infty\int_{X}| F(x,t)|_*\,| G(x,t)|_*\,d\mu(x)\,t\,dt\,\leq 6\,(p^*-1)\|\Pi f\|_p\|{\bf g}\|_q.
\end{equation}
\end{thm}
\begin{remark}
The theorem can be slightly generalized, at least at a formal level. Namely in Theorem \ref{thm:bilem}, we do not need that $v_i=[\delta_i,\delta_i^*].$ It is enough to have any $v_i\ge 0$ and take $Q_t=e^{-tM_i}$ with $M_i=L+v_i.$ 
\end{remark}

Our main theorem is an immediate corollary of Proposition \ref{pro:Form1} and Theorem \ref{thm:bilem}.

\begin{proof}[Proof of Theorem \ref{thm:main}]
It is enough to prove that for each $f\in L^p$ and $g_i\in L^q,$ $i=1,\ldots,d,$ the absolute value of 
$\sum_{i=1}^d \sk{R_i f}{g_i}$ does not exceed $$24(1+\sqrt{K})(p^*-1)\|\Pi f\|_p\,\bigg\|\bigg(\sum_{i=1}^d|g_i|^2\bigg)^{1/2}\bigg\|_{q}.$$ 
  	A density argument based on the assumption \eqref{eq:T3} allows us to take $f\in \mD$ and $g_i\in \mD_i,$ $i=1,\ldots,d.$  From Proposition \ref{pro:Form1} we have
	$$
	-\frac14\sk{R_i f}{g_i}_{L^2}=\int_0^\infty\sk{\d_i P_t\Pi f}{\pd_t \tPi g_i}_{L^2}\,t\,dt+\int_0^\infty\sk{q_i P_t\Pi f}{\pd_t \tPi g_i}_{L^2}\,t\,dt$$ and thus, assumption \eqref{eq:A2} gives
\begin{align*}&\left|\sum_{i=1}^d\sk{R_i f}{g_i}_{L^2}\right|\\
&\leq 4\int_0^\infty\int_{X}\left( \bigg(\sum_{i=1}^d|\d_iP_t\Pi f(x)|^2\bigg)^{1/2}+\sqrt{K} \sqrt{r(x)}|P_t\Pi f(x)|\right)\,| G(x,t)|_*\,d\mu(x)\,t \,dt\\
& \le 4(1+\sqrt{K})\int_0^\infty\int_{X}|F(x,t)|_*\,| G(x,t)|_*\,d\mu(x)\,t \,dt.\end{align*}
	Now, Theorem \ref{thm:bilem} completes the proof.
\end{proof}


\section{Bilinear embedding theorem}
\label{sec:BilEm}
This section is devoted to the proof of our embedding theorem - Theorem \ref{thm:bilem}. We shall follow closely the reasoning from \cite{CD} and \cite{MSJFA}.

\subsection{The Bellman function}

Before proceeding to the proof of Theorem \ref{thm:bilem} we need to introduce its most important ingredient: the Bellman  function. 


Choose $p\geq 2$. Let $q=p/(p-1)$,
$$
\gamma=\gamma(p)=\frac{q(q-1)}{8},
$$
and define $\beta_{p}\colon [0,\infty)^2\to [0,\infty]$ by
\begin{equation*}
\beta_{p}(s_1,s_2)
=s_1^{p}+s_2^{q}+\gamma
\left\{
\aligned
& s_1^2\,s_2^{2-q} & ; & \ \ s_1^p\leq s_2^q\\
& \frac{2}{p}\,s_1^{p}+\left(\frac{2}{q}-1\right)s_2^{q}
& ; &\ \ s_1^p\geq s_2^q\,.
\endaligned\right.
\end{equation*}
For $m=(m_1,m_2)\in \N^2$ the Nazarov-Treil Bellman function corresponding to $p,m$ is the function
$$
B=B_{p,m}\colon \R^{m_1}\times\R^{m_2}\rightarrow[0,\infty)
$$
given, for any $\zeta\in\R^{m_1}$ and $\eta\in\R^{{m_2}}$, by
\begin{equation*}
B_{p,m}(\zeta,\eta)=\frac12\beta_p(|\zeta|,|\eta|).
\end{equation*}
The function $B$ bears its origins in the article  \cite{NT} by F. Nazarov and S. Treil.
It was employed (and simplified) in \cite{CD,CD-mult,DV,DV-Kato,DV-Sch}. Note that $B$ is $C^1(\R^{m_1+m_2})$ and is $C^2$ everywhere except on the set
$$\{(\zeta,\eta)\in \R^{m_1}\times \R^{m_2}\colon \eta=0\textrm{ or }|\zeta|^p=|\eta|^q\}.$$ To remedy the non-smoothness of $B$ we consider the regularization $$B_{\kappa,p,m}=B_{\kappa}:=B*_{\R^{m_1+m_2}}\psi_{\kappa},$$ where $$\psi(x)=c_{m}e^{-\frac1{1-|x|^2}}\chi_{B^{m_1+m_2}}(x)\quad\textrm{and}\quad \psi_\kappa(x)=\frac1{\kappa^{m_1+m_2}}\psi(x/\kappa),$$
with $c_{m}$ such that $\int_{\R^{m_1+m_2}} \psi_{\kappa}(x)\,dx=1.$ Here $\chi_{B^{m_1+m_2}}$ stands for the characteristic function of the $(m_1+m_2)$--dimensional Euclidean ball centered at the origin and of radius $1.$
Since both $B$ and $\psi_{\kappa}$ are bi-radial also $B_{\kappa}$ is bi-radial. Hence, there is $\beta_{\kappa}=\beta_{\kappa,p}$ acting from $[0,\infty)^2$ to $\R$ such that
$$B_{\kappa}(\zeta,\eta)=\frac12\beta_{\kappa}(|\zeta|,|\eta|),\qquad \zeta\in \R^{m_1}, \eta \in \R^{m_2}.$$ 
We shall need some properties of $\beta_{\kappa}$ and $B_{\kappa}$ that were essentially proved in \cite{CD}, \cite{DV-Sch}, and \cite{MS}, \cite{MSJFA}. 
\begin{prop} Let $\kappa\in (0,1).$ Then, for  $s_i>0,$ $i=1,2,$ we have
	\label{pro:Bkprop}
	
	\begin{enumerate}[{\rm (i)}]
		\item
		\label{anterija}
		$
		0\leqslant \beta_{\kappa}(s_1,s_2)\leqslant (1+\gamma(p))\left((s_1+\kappa)^p+(s_2+\kappa)^q\right),
		$ 
		\item 	$0\leq \pd_{s_1}\beta_{\kappa}(s)\leq C_p\max( (s_1+\kappa)^{p-1},s_2+\kappa)$ and $0\leq \pd_{s_2}\beta_{\kappa}(s)\leq C_p(s_2+\kappa)^{q-1},$ with $C_p$ being a positive constant.
	\end{enumerate}
The function $B_{\kappa}$  belongs to $C^{\infty}(\R^{m_1+m_2}),$ and for any $\xi=(\zeta,\eta)\in \R^{m_1+m_2}$ there exists a positive $\tau_{\kappa}=\tau_{\kappa}(|\zeta|,|\eta|)$ such that for $\omega=(\omega_1,\omega_2)\in \R^{m_1+m_2}$ we have
	\begin{enumerate}[{\rm (i)}]
	\addtocounter{enumi}{2}
	\item \label{kupres}
	$
	{\sk{\Hess(B_{\kappa})(\xi)\omega}{\omega} \geqslant \frac{\gamma(p)}{2}\big(\tau_{\kappa}|\omega_1|^2+\tau_{\kappa}^{-1}|\omega_2|^2\big)}$.
	\end{enumerate}
Moreover, there is a continuous function $E_{\kappa}\colon \R^{m_1+m_2}\to \R$ for which
	\begin{enumerate}[{\rm (i)}]
\addtocounter{enumi}{3}
	\item \label{kupres2}
	$
	{\sk{(\nabla B_{\kappa})(\xi)}{\xi} \geqslant \frac{\gamma(p)}{2}\big(\tau_{\kappa}|\zeta|^2+\tau_{\kappa}^{-1}|\eta|^2\big)}-\kappa E_{\kappa}(\xi)$,
	\item 
	\label{kupres3}
$|E_{\kappa}(\xi)|\le C_{m,p}(|\zeta|^{p-1}+|\eta|+|\eta|^{q-1}+\kappa^{q-1}).$
\end{enumerate}
\end{prop}
\begin{proof}[Proof (sketch)]
	Let $\tau=\tau(|\zeta|,|\eta|)$ be the function from \cite[Theorem 3]{CD} and define $\tau_{\kappa}=\tau*_{\R^{d+1}}\psi_{\kappa}.$ With exactly this $\tau_{\kappa}$ items i), ii), and iii) were proved in \cite[Proposition 6.3]{MS}. 
	
	Let $$E_{\kappa}(\xi)=-\int_{\R^{m_1+m_2}} \sk{\nabla B(\xi-\kappa s)}{s}\,\psi_{\kappa}(s)\,ds,\qquad \xi\in \R^{m_1+m_2}$$ cf. \cite[eq. (2.10)]{DV-Sch}.    Item iv) (with these $\tau_{\kappa}$ and $E_{\kappa}$) follows from \cite[Theorem 4 iii')]{DV-Sch}, together with the observation from \cite{CD,DV-Sch}
	that
	$$(\tau*\psi_{\kappa})(\xi)(\tau^{-1}*\psi_{\kappa})(\xi)\ge \left(\int_{\R^{d+1}}\left(\tau(y)\psi_{\kappa}(x-y)\right)^{1/2}\left(\tau^{-1}(y)\psi_{\kappa}(x-y)\right)^{1/2}\,dy\right)^{1/2}=1.$$
Item v) is proved in \cite[p. 207]{DV-Sch}. Note that, our Bellman function $B_{\kappa}$ coincides with $-\frac12 Q_{\kappa}$ from \cite{DV-Sch} (when $Q_{\kappa}$ is restricted to real arguments).

We remark that  in  \cite[Theorem 4 iii')]{DV-Sch} a stronger statement is proved with an additional negative term $-B_{\kappa}(\xi)$ on the left hand side of iv).
\end{proof}

\subsection{Proof of Theorem \ref{thm:bilem}.}
\label{tresenica}
Define $u:X\times(0,\infty)\rightarrow\R\times\R^d$ by
$$
u=u(x,t)=\big(P_t\,\Pi f(x),Q_t{\bf g}(x)\big)=\big(P_t\,\Pi f(x),Q_t^1 g_1(x),\ldots,Q_t^d g_d(x)\big).
$$

Assume first that $p\ge 2$ and set
$$
b_{\kappa}=B_{\kappa}\circ u: X\times(0,\infty)\rightarrow [0,\infty).
$$
Here $B_{\kappa}=B_{\kappa,d,p}$ is the Bellman function from Proposition \ref{pro:Bkprop} with $m_1=1$ and $m_2=d.$  
For each $i=1,\ldots,d,$ we fix a sequence $\{\sigma_i^n\}_{n\in \N}$ which converges to $\sigma_i,$ and a sequence $\{\Sigma_i^n\}_{n\in \N}$ which converges to $\Sigma_i.$ We also impose that $\sigma_i<\sigma_i^n<\Sigma_i^n<\Sigma_i,$ for $i=1,\ldots,d,$ $n\in \N.$      
Defining $$X_i^{n}=[\sigma_i^n,\Sigma_i^n]\qquad\textrm{and}\qquad X_n=X_1^n\times\cdots \times X_d^n,$$ where $n\in \N,$ we see that  $\{X_n\}_{n\in \N}$ is an increasing family of compact subsets of $X$ such that $X=\bigcup_{n} X_n.$  We shall estimate the integral
\begin{equation}
\label{eq:4}
I(n,\varepsilon):=\int_{0}^{\infty}\int_{X_n} (\pd_{t}^2-\tL) (b_{\kappa(n)})(x,t)\,d\mu(x)\,te^{-\varepsilon t}dt
\end{equation}
from below and above and then, first let  $n\to \infty$ and then $\varepsilon\to 0^+.$ Here $\kappa(n)$ is a small quantity depending on $n$ which will be determined in the proof.  Since $X_n$ is compact, $f\in \mD$ and $g_i\in \mD_i,$ $i=1,\ldots,d,$ the integral \eqref{eq:4} is in fact absolutely convergent. In what follows we will often briefly write $\kappa$ instead of $\kappa(n).$ 

{\bf The lower estimate of \eqref{eq:4} for $p\ge2$.} The key result here is Proposition \ref{prop:formula1} below. Its proof hinges on the assumption \eqref{eq:A1}.
\begin{prop}
	\label{prop:formula1}
For $x\in X$ and $t>0$ it holds
	\begin{equation}
	\label{eq:7}
	\left(\left(\pd_{t}^2-\tL\right)b_{\kappa}\right)(x,t)\geq \gamma |F(x,t)|_*|G(x,t)|_*-\kappa\, r(x) E_{\kappa}(u(x,t)).
	\end{equation}

\end{prop}

	\begin{proof}
			Set $\d_0:=\pd_t.$ To justify \eqref{eq:7} we shall need the pointwise equality
			\begin{equation}
			\label{eq:8}
			\left(\pd_{t}^2-\Ld\right)b_{\kappa}
			=r\,\sk{\nabla B_{\kappa}(u)}{u}+\sum_{i=1}^dv_i\cdot(\pd_{\eta_i}B_{\kappa}(u)\cdot Q_t^{i}g_i)
			+\sum_{i=0}^d\sk{\Hess(B_{\kappa})(\d_i u)}{\d_i u}.
			\end{equation}
			First we focus on proving \eqref{eq:8}.
		
		From the chain rule we have
		$\d_ib_{\kappa}=p_i\sk{\nabla B_{\kappa}(u)}{\partial_{x_i}u}.$
		Moreover, a computation shows that, for $i=1,\ldots,d,$
		$$\d_i^*=-p_i\partial_{x_i}-p_i\frac{w_i'}{w_i}-p_i',\quad\textrm{and}\quad\d_i^*\d_i=-p_i^2\partial^2_{x_i}-\big(p_i\frac{w_i'}{w_i}+2p_i'\big)p_i\partial_{x_i}.$$
		Consequently, applying once again the chain rule we obtain for $i=0,\ldots,d,$ 
		\begin{align*}
		\d_i^*\d_ib_{\kappa}&=-p_i\partial_{x_i}(p_i\sk{\nabla B_{\kappa}(u)}{\partial_{x_i}u})-\big(p_i\frac{w_i'}{w_i}+p_i'\big)p_i\sk{\nabla B_{\kappa}(u)}{\partial_{x_i}u}\\
		&=-p_i^2\partial_{x_i}(\sk{\nabla B_{\kappa}(u)}{\partial_{x_i}u})-p_ip_i'\sk{\nabla B_{\kappa}(u)}{\partial_{x_i}u}-\big(p_i\frac{w_i'}{w_i}+p_i'\big)p_i\sk{\nabla B_{\kappa}(u)}{\partial_{x_i}u}\\
		&=\sk{\nabla B_{\kappa}(u)}{-p_i^2\partial^2_{x_i}u-\big(p_i\frac{w_i'}{w_i}+2p_i'\big)p_i\partial_{x_i}u}-p_i^2\sk{\Hess (B_{\kappa})(\partial_{x_i} u)}{\partial_{x_i} u}\\
		&=\sk{\nabla B_{\kappa}(u)}{\d_i^*\d_i u}
		-\sk{\Hess (B_{\kappa})(\d_i u)}{\d_i u}.
		\end{align*}
		Now, summing the above formula in $i=0,\ldots,d,$ we obtain 
	\begin{equation}
	\label{eq:formula1}
		(\d_0^2-\Ld) b_{\kappa}
		=\sk{\nabla B_{\kappa}(u)}{(\d_0^2-\Ld) u}
		+\sum_{i=0}^d\sk{\Hess(B_{\kappa})(u)(\d_i u)}{\d_i u}.
	\end{equation}
	
	The formula \eqref{eq:formula1} implies \eqref{eq:8}. Indeed
	we have
	$$
	\left(\pd_t^2-L\right)u=
	\big((\pd_{t}^2-L)P_tf,(\pd_{t}^2-L)Q_t{\bf g}
	\big),
	$$
	where
	$$
	\left(\pd_{t}^2-L\right)P_tf = 0
	$$
	and
	$$
	\left(\pd_{t}^2-L\right)Q_t{\bf g} =
	\left(
	\left(\pd_{t}^2-L\right)Q_t^{1}g_1,
	\hdots,
	\left(\pd_{t}^2-L\right)Q_t^{d}g_d
	\right)\,.
	$$
	Moreover,
	$$
	\left(\pd_{t}^2-L\right)Q_t^{i}g_i=
	\left(\pd_{t}^2-M_i\right)Q_t^{i}g_i + v_i\cdot Q_t^{i}g_i=v_i\cdot  Q_t^i g_i,
	$$
	and using \eqref{eq:formula1} the equation \eqref{eq:8} follows.
	
	Having demonstrated \eqref{eq:8} we pass to the proof of \eqref{eq:7}. Item ii) of Proposition \ref{pro:Bkprop} implies $(\pd_{\eta_i}B_{\kappa}(u)\cdot Q_t^{i}g_i)\ge 0$. Thus \eqref{eq:8} together with the assumption \eqref{eq:A1} produce
		\begin{equation}
	\label{eq:6}
	\left(\pd_{t}^2-\tL\right)b_{\kappa}\geq r\,\sk{\nabla B_{\kappa}(u)}{u}+\sum_{i=0}^d\sk{\Hess (B_{\kappa})(\d_i u)}{\d_i u}.
	\end{equation}
	Finally, \eqref{eq:7} is a consequence of \eqref{eq:6}, items iii) and iv) from Proposition \ref{pro:Bkprop}, and the inequality between the arithmetic and geometric mean. 		
	\end{proof}

Coming back to the proof of the lower estimate in \eqref{eq:4} we now take $\{\kappa(n)\}_{n\in\N}$ such that $|\kappa(n)|\leq 1,$ $\lim_n\kappa(n)=0$ and 
\begin{equation}
\label{eq:k1}
|\kappa(n)|^{1/2}\int_{X_n}|r(x)E_{\kappa(n)}(u(x,t))|\,d\mu(x)\leq 1.
\end{equation} To see that such a sequence exists we use Proposition \ref{pro:Bkprop} v) and the fact that $P_t f\in \mD$ and $Q_t^i g_i\in \mD_i$ (hence also $P_tf\in C^{\infty}(X)$ and $Q_t^i g_i \in C^{\infty}(X)$). Next, \eqref{eq:7} together with \eqref{eq:k1} lead to
\begin{equation*}
\liminf_{n\to \infty} I(n,\varepsilon)\ge \gamma \int_0^\infty\int_{X}| F(x,t)|_*\,| G(x,t)|_*\,d\mu(x)\,te^{-\varepsilon t}\,dt,
\end{equation*}
and, consequently, by the monotone convergence theorem
\begin{equation}
\label{eq:10}
\liminf_{\varepsilon\to 0^+}\liminf_{n\to \infty} I(n,\varepsilon)\ge \gamma(p) \int_0^\infty\int_{X}| F(x,t)|_*\,| G(x,t)|_*\,d\mu(x)\,t\,dt.
\end{equation}
This is our lower estimate of \eqref{eq:4}.

\bigskip
{\bf The upper estimate of \eqref{eq:4} for $p\ge2$}. The main ingredients here are the technical assumptions \eqref{eq:T2} and \eqref{eq:T3}. We split the integral in  \eqref{eq:4} as \begin{align*}I(n,\varepsilon)&=I_1(n,\varepsilon)-I_2(n,\varepsilon)\\ &:=\int_{0}^{\infty}\int_{X_n} \pd_{t}^2 (b_{\kappa(n)})(x,t)\,d\mu(x)\,te^{-\varepsilon t}dt-\int_{0}^{\infty}\int_{X_n} \tL (b_{\kappa(n)})(x,t)\,d\mu(x)\,te^{-\varepsilon t}dt.
&\end{align*}

First we prove that 
\begin{equation}
\label{eq:12} \lim_{n\to \infty}I_2(n,\varepsilon)=0.
\end{equation} 
To see this we recall that $\tL=\sum_{i=1}^d\d_i^*\d_i$ with $\d_i$ given by \eqref{eq:didef} and $\d_i^*$ being the formal adjoint of $\d_i$ on $L^2.$ Then, 
$$I_2(n,\varepsilon)=\sum_{i=1}^d I_2^i(n,\varepsilon):=\sum_{i=1}^d\int_{0}^{\infty}\int_{X_n} (\d_i^*\d_i) (b_{\kappa(n)})(x,t)\,d\mu(x)\,te^{-\varepsilon t}dt,$$ and it is enough to prove that each of the integrals $I_2^i(n,\varepsilon)$ goes to $0$ as $n\to \infty.$ As the reasoning is symmetric in $i=1,\ldots,d,$ we present it only for $I_2^1(n,\varepsilon).$ Denote $$X^{(1)}=X_2\times\cdots \times X_d,\quad x^{(1)}=(x_2,\ldots,x_d),\quad \textrm{and }\mu^{(1)}=\mu_2\otimes\cdots \otimes \mu_d.$$  Formula \eqref{eq:dderadj} together with integration by parts in the $x_1$ variable produces 
$$I_2^1(n,\varepsilon)=\int_0^{\infty}\int_{X^{(1)}}\left((p_1^2w_1\partial_{x_1}b_{\kappa})(\Sigma_1^n,x^{(1)})-(p_1^2w_1\partial_{x_1}b_{\kappa})(\sigma_1^n,x^{(1)})\right)\,d\mu^{(1)}(x^{1})\,te^{-\varepsilon t}\,dt.$$ 

Call $z_1^n$ any of the quantities $\sigma_1^n$ or $\Sigma_1^n.$ Then the chain rule gives
\begin{equation}
\label{eq:14}
\begin{split}&(p_1^2w_1\partial_{x_1}b_{\kappa})(z_1^n,x^{(1)})\\
&=p_1^2(z_1^n) w_1(z_1^n) \pd_{x_1}P_tf(z_1^n,x^{(1)})\pd_{\zeta}B_{\kappa}\big(P_t f(z_1^n,x^{(1)}),Q_t {\bf g}(z_1^n,x^{(1)})\big)\\
&+p_1^2(z_1^n) w_1(z_1^n) \sk{\pd_{x_1}Q_t{\bf g}(z_1^n,x^{(1)})}{\nabla_{\eta}B_{\kappa}\big(P_t f(z_1^n,x^{(1)}),Q_t {\bf g}(z_1^n,x^{(1)})\big)}.\end{split}\end{equation}  
Since $f\in \mD$ and $g_i \in \mD_i$ we have that $P_t f\in \mD$ and $Q_t{\bf g}\in \mD_1\otimes\cdots \otimes \mD_d.$ Recall that $\fk$ is defined by \eqref{eq:fkdef} while  $\delta_i \fk,$ $i=1,\ldots,d,$ are given by \eqref{eq:dfkidef}. Now, Proposition \ref{pro:Bkprop} ii) implies
\begin{equation}
\label{eq:Pro5ii}
|\nabla_{\zeta,\eta}B_{\kappa}(\zeta,\eta)|\leq C_{p,q}(|\zeta|^{p-1}+|\eta|^{q-1}+|\eta|+\kappa^{q-1}).\end{equation}
Therefore, since $|\kappa(n)|\leq 1,$ a calculation based on \eqref{eq:14} together with the assumptions \eqref{eq:T2}, \eqref{eq:T3}, and H\"older's inequality produces $\lim_{n}I_2^1(n,\varepsilon)=0.$

Now we focus on $I_1(n,\varepsilon).$ Since $f\in \mD,$ $g_i\in \mD_i,$ $i=1,\ldots,d,$ $B_{\kappa}\in C^{\infty}(\R^{d+1})$ and we integrate over $x\in X_n,$ the double integral is absolutely convergent. Thus Fubini's theorem gives
\begin{equation*}
I_1(n,\varepsilon)=\int_{X_n} \int_{0}^{\infty}\pd_{t}^2 (b_{\kappa(n)})(x,t)\,te^{-\varepsilon t}dt\,d\mu(x).
\end{equation*}
Integrating by parts in the inner integral twice we obtain
\begin{align*}
&I_1(n,\varepsilon)=-\int_{X_n}\int_0^{\infty}\pd_{t} (b_{\kappa(n)})(x,t)(1-\varepsilon t)e^{-\varepsilon t}\,dt\,d\mu(x)=\int_{X_n}b_{\kappa(n)}(x,0)\,d\mu(x)\\
&+\varepsilon^2 \int_{X_n}\int_0^{\infty}b_{\kappa(n)}(x,t)\,t e^{-\varepsilon t}\,dt\,d\mu(x)-2\varepsilon \int_{X_n}\int_0^{\infty}b_{\kappa(n)}(x,t)e^{-\varepsilon t}\,dtd\mu(x)\\
&\leq \int_{X_n}b_{\kappa(n)}(x,0)\,d\mu(x)+\varepsilon^2 \int_{X_n}\int_0^{\infty}b_{\kappa(n)}(x,t)\,t e^{-\varepsilon t}\,dt\,d\mu(x)\\
&:=I_1^1(n)+I_1^2(n,\varepsilon).
\end{align*}
In the first two equalities above we neglected the boundary terms by using the chain rule together with \eqref{eq:Pro5ii}. 

First we treat $I_1^2(n,\varepsilon).$ Proposition \ref{pro:Bkprop} i) gives
$$I_1^2(n,\varepsilon)\leq \varepsilon^2 C_{p}\int_{X_n}\int_0^{\infty}\left(|P_t \Pi f(x)|^p+|Q_t{\bf g}(x)|^q+\max(\kappa(n)^p,\kappa(n)^q)\right)\,te^{-\varepsilon t}dt\,d\mu(x).$$
Take $\kappa(n)$ which satisfies \eqref{eq:k1} and
\begin{equation}
\label{eq:k2}
\max(\kappa(n)^{p-1/2},\kappa(n)^{q-1/2})\,\mu(X_n)\leq 1.
\end{equation}
Then, since $f\in \mD$ and $g_i\in \mD_i,$ $i=1,\ldots,d,$ we have
$$\limsup_{n\to\infty} I_1^2(n,\varepsilon)\leq \varepsilon^2 C_{p}\int_{X}\int_0^{\infty}|P_t \Pi f(x)|^p+|Q_t{\bf g}(x)|^q\,tdt\,d\mu(x)\leq C(p,f,g)\, \varepsilon^2,$$
and, consequently,
\begin{equation}
\label{eq:-14}
\limsup_{\varepsilon\to 0^+}\limsup_{n\to\infty} I_1^2(n,\varepsilon)=0.
\end{equation}

Coming back to $I_1^1(n)$ we use Proposition \ref{pro:Bkprop} i) to estimate
$$I_1^1(n)\leq \frac{1+\gamma}{2}\int_{X_n}(|\Pi f(x)|+\kappa(n))^p\,d\mu(x)+\frac{1+\gamma}{2}\int_{X_n}(|{\bf g}(x)|+\kappa(n))^q\, d\mu(x).$$
Now for each $\varepsilon>0$ we split the first integral onto $\int_{|\kappa(n)|\leq \varepsilon |\Pi f(x)|} $ and $\int_{|\kappa(n)|> \varepsilon |\Pi f(x)|}$ and the second integral onto $\int_{|\kappa(n)|\leq \varepsilon |{\bf g}(x)|} $ and $\int_{|\kappa(n)|> \varepsilon |{\bf g}(x)|}.$ Then we obtain
\begin{align*}
I_1^1(n)\leq& \frac{1+\gamma}{2}\left((1+\varepsilon)^p\|\Pi f\|_p^p+(1+\varepsilon)^q\|{\bf g}\|_q^q\right)\\&+\frac{1+\gamma}{2}\left((1+\varepsilon^{-1})^{p}\kappa(n)^p\mu(X_n)+(1+\varepsilon^{-1})^{q}\kappa(n)^q\mu(X_n)\right).
\end{align*}   
Since $\kappa(n)$ satisfies \eqref{eq:k1} and \eqref{eq:k2} we arrive at
$$\limsup_{\varepsilon\to 0^+}\limsup_{n\to\infty} I_1^1(n)\leq \frac{1+\gamma}{2}\left(\|\Pi f\|_p^p+\|{\bf g}\|_q^q\right).$$
Recalling \eqref{eq:12} and \eqref{eq:-14} we thus proved
\begin{equation}
\label{eq:15}
\limsup_{\varepsilon\to 0^+}\limsup_{n\to\infty} I(n,\varepsilon)\leq \frac{1+\gamma(p)}{2}\left(\|\Pi f\|_p^p+\|{\bf g}\|_q^q\right), \end{equation}
which is the upper estimate of \eqref{eq:4} we need.

\bigskip
{\bf Completion of the proof of the bilinear embedding}. Consider first $p\ge 2.$ Combining the lower estimate \eqref{eq:10} and the upper estimate \eqref{eq:15} we obtain 
\begin{equation}
\label{eq:17}
\int_0^\infty\int_{X}| F(x,t)|_*\,| G(x,t)|_*\,d\mu(x)\,t\,dt\,\leq \frac{1+\gamma(p)}{2\gamma(p)}\bigg(\|\Pi f\|_p^p+\|{\bf g}\|_q^q\bigg).
\end{equation}
Finally, a polarization arguments finishes the proof. More precisely, for $s>0$ we replace $f$ with $s f$ and ${\bf g}$ with $s^{-1}{\bf g}$ on both sides of \eqref{eq:17}. Then, the left hand side is unchanged, while minimizing the right hand side over $s>0$ we obtain
\begin{equation}
\label{eq:ske0}
\int_0^\infty\int_{X}| F(x,t)|_*\,| G(x,t)|_*\,d\mu(x)\,t\,dt\,\leq \frac{1+\gamma(p)}{2\gamma(p)}\left(\bigg(\frac pq\bigg)^{1/p}+\bigg(\frac qp\bigg)^{1/q}\right)\|\Pi f\|_p\|{\bf g}\|_q.
\end{equation} 
Using the above inequality, a calculation leads to \eqref{eq:18}. We sketch the argument below.

Note that for $p\ge 2$ we have $p^*=p$ and recall that $\gamma(p)=q(q-1)/8.$ Thus, for  $1<q\le 2$  we obtain 
\begin{equation}
\label{eq:ske1}
\begin{split}
\frac{1+\gamma(p)}{2\gamma(p)}\left(\bigg(\frac pq\bigg)^{1/p}+\bigg(\frac qp\bigg)^{1/q}\right)&=\frac{8+q(q-1)}{2}(q-1)^{1/q-1}(p-1)\\
&\le (q+3) (q-1)^{1/q-1}(p^*-1).
\end{split}\end{equation}
Denoting $s=q-1$ we need to maximize the function $H(s):=(s+4)s^{-s/(s+1)}$ for $s\in(0,1].$ Let 
$$h(s)=\log(s+4)-\frac{s \log s}{s+1},$$
so that $H(s)=e^{h(s)}.$ 
Then we have
$$h'(s)=\frac{1}{s+4}-\frac{\log s }{(s+1)^2}-\frac{1}{s+1}\quad \textrm{and}\quad h''(s)=-\frac1{(s+4)^2}+\frac{2\log s }{(s+1)^3}+\frac{s-1}{s(s+1)^2};$$
consequently, $h''(s)<0$ for $s\in(0,1).$  Observe that $h'(7/20)>0$ and $h'(2/5)<0.$ Therefore $h'$ has a unique zero inside the interval $(7/20,2/5)$ and $h$ attains a global maximum there. Obviously, the same is true for $H=e^h.$ Now it is easy to see that
 $$\max_{7/20\le s\le 2/5}H(s)<\frac{22}{5}\cdot(7/20)^{-2/7}
<6,$$ and thus also $\sup_{0<s\le 1}H(s)<6.$ 
 Hence, coming back to \eqref{eq:ske1} we obtain
 $$\frac{1+\gamma(p)}{2\gamma(p)}\left(\bigg(\frac pq\bigg)^{1/p}+\bigg(\frac qp\bigg)^{1/q}\right) \le 6\, (p^*-1).$$
 In view of \eqref{eq:ske0} this implies \eqref{eq:18} and completes the proof of Theorem \ref{thm:bilem} for $p\ge 2$.

The proof of Theorem \ref{thm:bilem} for $p\le 2$ proceeds analogously once we switch $p$ with $q$ and $P_t f$ with $Q_t{\bf g}$ in the definition of $b_{\kappa}.$ Namely, we consider $\tilde{b}_\kappa(x,t)=\tilde{B}_{\kappa}(Q_t{\bf g},P_t f)$ where $\tilde{B}_{\kappa}(\zeta,\eta)=B_{\kappa,q,(d,1)}(\zeta,\eta),$ $\zeta\in \R^d,$ $\eta \in \R.$ Here $B_{\kappa,q,(d,1)}$ is the function from Proposition \ref{pro:Bkprop} with $m_1=d$ and $m_2=1.$ Then we repeat the argument used for $p\ge 2.$ The function $\tilde{B}_{\kappa}$ satisfies items iii)-v) of Proposition \ref{pro:Bkprop} with $p$ replaced by $q$.  Therefore both the lower estimate \eqref{eq:10} and the upper estimate \eqref{eq:15} hold with $\gamma(p)$ replaced by $\gamma(q).$

\section{Examples}
\label{sec:Exa}

Throughout this section we apply Theorem \ref{thm:main} to the examples of orthogonal systems considered by Nowak and Stempak in \cite[Section 7]{NS2}. This is possible for all of these systems except for the Fourier-Bessel expansions \cite[Section 7.8]{NS2}. In this case the condition \eqref{eq:T2} fails. Despite this failure we think that it might be possible to treat also the Fourier-Bessel expansions by the methods of the present paper. It might be also interesting to try to apply the methods of our paper to the Riesz transforms  considered by Nowak and Sj\"ogren in \cite{NSj3} (in the case of Jacobi trigonometric polynomial expansions).

In all of the examples we present, for more details the reader is kindly referred to \cite[Sections 7.1-7.7]{NS2}. The formulas for $v_i$ and $r=\sum_{i=1}^d r_i$ in the examples below follow directly from \eqref{eq:com1} and \eqref{eq:rform1}. Recall that 
$$\mu=\mu_1\otimes\cdots \otimes \mu_d,\qquad X=X_1\times\cdots\times X_d,\qquad L^p=L^p(X,\mu),\qquad \|\cdot\|_p=\|\cdot\|_{L^p},$$
and
$$p^*=\max\bigg(p,\frac{p}{p-1}\bigg).$$
\subsection{Ornstein-Uhlenbeck Operator -  Hermite polynomial expansions}

\label{OUop}

Here we consider   $$p_i=1,\quad q_i=0 ,\quad a_i=0,\quad  w_i(x_i)=\pi^{-1/2}e^{-x_i^2},\quad d\mu_i(x_i)=w_i(x_i)\,dx_i,$$
on $X_i=\R.$ 
Then \begin{equation}
\label{eq:quanOU}
\delta_i=\d_i=\partial_{x_i},\quad\delta_i^*=-\partial_{x_i}+2x_i,\qquad v_i=[\delta_i,\delta_i^*]=2,\qquad r=0,\end{equation}
and
$$L=\sum_{i=1}^dL_i=-\Delta+2\sk{x}{\nabla}$$
is the Ornstein-Uhlenbeck operator on $X=\R^d.$ The operator $L$ is essentially self-adjoint on $C_c^{\infty}(\R^d)$ with the self-adjoint extension given by
$$Lf=\sum_{k\in \N^d}|k|\sk{f}{\tilde{H}_{k}}_{L^2}\tilde{H}_{k}.$$
In the formula above $|k|=k_1+\cdots +k_d,$ the symbol $L^2$ stands for $L^2=L^2(\R^d,\mu),$ while $\{\tilde{H}_k\}_{k\in \N^d}$ is the system of $L^2$ normalized Hermite polynomials, see \cite[Section 7.1]{NS2} and \cite[p. 60]{Leb}. In this section we take $$\fk=\tilde{H}_k,\qquad k\in \N^d.$$ 

Note that $\mu$ is a probability measure in this setting. The projection $\Pi$ becomes
$$\Pi f=\sum_{k\in \N^d,k\neq 0}\sk{f}{\tilde{H}_{k}}_{L^2}\tilde{H}_{k},\qquad f\in L^2.$$
Then $$(I-\Pi)f=\sk{f}{\tilde{H}_{0}}_{L^2}\tilde{H}_{0},$$ 
and, since $\tilde{H}_0=1$, the operator $I-\Pi$ is the projection onto the constants given by
$$(I-\Pi)f(x)=\int_{X}f(y)\,d\mu(y),\qquad x\in X.$$
Hence, by Holder's inequality $\|(I-\Pi)f\|_p\leq \|f\|_p,$ and, consequently,
\begin{equation}
\label{eq:PiOU}
\|\Pi f\|_p \leq 2\|f\|_p,\qquad 1\leq p\leq \infty.
\end{equation}   

Next
\begin{equation}
\label{eq:derOU}
\delta_i \tilde{H}_k=\sqrt{2k_j}\tilde{H}_{k-e_j},\end{equation}
where, by convention $\tilde{H}_{k-e_j}=0$ if $k_j=0.$ This convention is also used for the examples presented in the next sections. The Riesz transform is defined by
$$R_i f=\sum_{k\in \N^d,k\neq 0}\bigg(\frac{k_j}{|k|}\bigg)^{1/2}\sk{f}{\tilde{H}_{k}}_{L^2}\tilde{H}_{k-e_i},\qquad f\in L^2.$$  

Dimension-free estimates for the vector ${\bf R}f=(R_1f,\ldots,R_d f)$ were proved by Meyer \cite{Mey1} (see also \cite{Gun1}, \cite{Gut1}, and \cite{Pis1} for different proofs). Later Dragi\v{c}evi\'c and Volberg \cite[Corollary 0.4]{DV} found a proof which uses the Bellman function method. The best result in terms of the size of the constants is due to Arcozzi \cite[Corollary 2.4]{Arc1} who proved that $\|{\bf R} f\|_p\leq 2(p^*-1)\|f\|_p,$ $1<p<\infty.$ An application of Theorem \ref{thm:main} produces similar, though weaker, bounds.
\begin{thm}
	\label{thm:OU}
Fix $1<p<\infty.$ Then, for $f\in L^p$ such that $\int_{X}f(y)\,d\mu(y)=0,$ we have 
	\begin{equation}
	\label{eq:ROU}
	\left\|{\bf R} f\right\|_p\leq 24(p^*-1)\|f\|_{p}.
	\end{equation}
\end{thm}
\begin{remark}
	Using \eqref{eq:PiOU} we may extend the bound \eqref{eq:ROU} to all $f\in L^p$ with $24$ being replaced by $48.$
\end{remark}
\begin{proof}
We apply Theorem \ref{thm:main}. In order to do so we need to check that its assumptions are satisfied.

By \eqref{eq:quanOU} we see that \eqref{eq:A1} and \eqref{eq:A2} (with $K=0$) hold. Condition \eqref{eq:T1} is proved by an easy calculation based on integration by parts. The assumption \eqref{eq:T2} is also straightforward. Finally, \eqref{eq:T3} follows from Lemma \ref{lem:BC1} and \eqref{eq:derOU}.

Now, if $\int_{X}f(y)\,d\mu(y)=0$ then $\Pi f=f.$  Thus, an application of Theorem \ref{thm:main} completes the proof.
\end{proof}

\subsection{Laguerre operator -  Laguerre polynomial expansions}
\label{LagPol}

Here, for a parameter $\alpha\in (-1,\infty)^d,$  we consider   $$p_i=\sqrt{x_i},\quad q_i=0 ,\quad a_i=0,\quad   w_i(x_i)=\frac1{\Gamma(\alpha_i+1)}\,x_i^{\alpha_i}e^{-x_i}dx_i,\quad d\mu_i(x_i)=w_i(x_i)\,dx_i,$$
on $X_i=(0,\infty).$ 
Then $\delta_i=\d_i=\sqrt{x_i}\,\partial_{x_i},$ and thus \begin{equation}
\label{eq:quanLaPol}
\delta_i^*=-\sqrt{x_i}\,\partial_{x_i}-\frac{\alpha_i+1/2}{\sqrt{x_i}}+\sqrt{x_i},\quad v_i=[\delta_i,\delta_i^*]=\frac{\alpha_i+1/2+x_i}{2x_i},\quad r=0.\end{equation}
In this case 
$$L=\sum_{i=1}^dL_i=\sum_{i=1}^{d} -x_i\partial_{x_i}^2-(\alpha_i+1-x_i)\partial_{x_i}$$
is the Laguerre operator on $X=(0,\infty)^d.$ It is  symmetric on $C_c^{\infty}((0,\infty)^d)$ and has a self-adjoint extension 
\begin{equation*}
Lf=\sum_{k\in \N^d}|k|\sk{f}{\tilde{L}_{k}^{\alpha}}_{L^2}\tilde{L}_{k}^{\alpha};\end{equation*}
here $L^2=L^2((0,\infty)^d,\mu),$ while $\{\tilde{L}_k^{\alpha}\}_{k\in \N^d}$ is the system of $L^2$ normalized Laguerre polynomials, see \cite[Section 7.2]{NS2} and \cite[p. 76]{Leb}. These Laguerre polynomials are our functions $\fk$ in this section, namely
$$\fk=\tilde{L}_{k}^{\alpha},\qquad k\in \N^d.$$ 


Next we have
\begin{equation}
\label{eq:diLaPol}
\delta_i\tilde{L}_{k}^{\alpha}=\sqrt{k_j}\sqrt{x_i}\tilde{L}_{k-e_i}^{\alpha+e_i},\end{equation}
while the projection $\Pi$ becomes
$$\Pi f=\sum_{k\in \N^d,k\neq 0}\sk{f}{\tilde{L}_{k}^{\alpha}}_{L^2}\tilde{L}_{k}^{\alpha},\qquad f\in L^2.$$
A repetition of the argument from the previous section shows that $\Pi f=f$ if and only if $\int_{X}f(y)\,d\mu(y)=0$ and 
\begin{equation}
\label{eq:PiLaPol}
\|\Pi f\|_p \leq 2\|f\|_p,\qquad 1\leq p\leq \infty.
\end{equation}

The Riesz transform is then given by
$$R_i f=\sum_{k\in \N^d,k\neq 0}\bigg(\frac{k_j}{|k|}\bigg)^{1/2}\sk{f}{\tilde{L}_{k}^{\alpha}}_{L^2}\sqrt{x_i}\tilde{L}_{k-e_i}^{\alpha+e_i},\qquad f\in L^2.$$  
Dimension-free bounds for single Riesz transforms $R_i$ were first studied by Guti\'errez, Incognito and Torrea \cite{GIT} (half-integer multi-indices), and generalized\footnote[2]{In \cite[Theorem 13]{No1} the author also states an estimate on $L^p$ for the vector of Riesz-Laguerre
	transforms that is dimension-free for certain values of $\alpha$. Unfortunately this result is not properly proved there \cite{Npriv}. This is due to a problem in the proof of the vectorial $g$-function bound from 
	\cite[Theorem 7(b)]{No1}.} by Nowak \cite{No1} (to multi-indices $\alpha\in[-1/2,\infty)^d$). Moreover in \cite{GLLNU}, Graczyk, Loeb, L\'opez, Nowak, and Urbina proved dimension-free estimates on $L^p$ for the vector of Riesz-Laguerre transforms and half-integer multi-indices $\alpha$. Recently, the author \cite[Theorem 4.1 b)]{BWr1} obtained dimension-free bounds on $L^p$ for scalar Riesz transforms and general parameters $\alpha\in(-1,\infty)^d,$ while Mauceri and Spinelli \cite[Theorem 5.2]{MSJFA} proved a dimension-free bound for the vectorial Riesz transforms ${\bf R} f=(R_1f,\ldots,R_df)$ (and $\alpha\in[-1/2,\infty)^d$). All the bounds mentioned in this paragraph are also independent of the parameter $\alpha$ (appropriately restricted). Moreover, the estimate from \cite[Theorem 5.2]{MSJFA} is also linear in $p^*$.

By using Theorem \ref{thm:main} we obtain a result which coincides with \cite[Theorem 5.2]{MSJFA} in the case of Riesz transforms acting on functions.
\begin{thm}
	\label{thm:LaPol}
	Fix $\alpha\in [-1/2,\infty)^d$ and $1<p<\infty.$ Then, for $f\in L^p$ which satisfy $\int_{X}f(y)\,d\mu(y)=0,$  we have 
	\begin{equation*}
	\left\|{\bf R} f\right\|_p\leq 24(p^*-1)\|f\|_{p}.
	\end{equation*}
\end{thm}
\begin{remark}
	By \eqref{eq:PiLaPol} we have the same bound for general $f\in L^p$ with the constant being twice as large.
\end{remark}
\begin{proof}
We are going to apply Theorem \ref{thm:main}, so we need to verify its assumptions.
	
	By \eqref{eq:quanLaPol} we see that if $\alpha\in [-1/2,\infty)^d,$ then  \eqref{eq:A1} and \eqref{eq:A2} (with $K=0$) are satisfied. Moreover, the assumptions \eqref{eq:T1} and \eqref{eq:T2} follow from a direct calculation. Next, for such $\alpha$ the condition \eqref{eq:T3} can be deduced from Lemma \ref{lem:BC1} together with \eqref{eq:diLaPol}. 
	
 Now, if $\int_{X}f(y)\,d\mu(y)=0$ then $\Pi f=f.$ Thus, using Theorem \ref{thm:main} we complete the proof of Theorem \ref{thm:LaPol}.
\end{proof}

\subsection{Jacobi operator -  Jacobi polynomial expansions}
\label{sec:Jac}

In this section for parameters $\alpha,\beta\in (-1,\infty)^d$  we consider   \begin{align*}p_i&=\sqrt{1-x_i^2},\quad q_i=0 ,\quad a_i=0,\\
 w_i(x_i)&=\frac1{C(\alpha_i,\beta_i)}\,(1-x_i)^{\alpha_i}(1+x_i)^{\beta_i}dx_i,\quad  d\mu_i(x_i)=w_i(x_i)dx_i,\quad X_i=(-1,1),\end{align*}
 where $C(\alpha_i,\beta_i)$ is such that $\mu_i(X_i)=1.$  
Then $\delta_i=\d_i=\sqrt{1-x_i^2}\,\partial_{x_i},$ and \begin{equation}
\label{eq:quanJaPol}
\begin{split}
\delta_i^*&=-\sqrt{1-x_i^2}\,\partial_{x_i}+(\alpha_i+1/2)\sqrt\frac{1+x_i}{1-x_i}-(\beta_i+1/2)\sqrt\frac{1-x_i}{1+x_i},\\
 v_i&=[\delta_i,\delta_i^*]=\frac{\alpha_i+1/2}{1-x_i}+\frac{\beta_i+1/2}{1+x_i},\qquad r=0.
\end{split}\end{equation}
Here
$$L=\sum_{i=1}^dL_i=\sum_{i=1}^{d} -(1-x_i^2)\partial_{x_i}^2-(\beta_i-\alpha_i-(\alpha_i+\beta_i+2)x_i)\partial_{x_i}$$
is the Jacobi operator on $X=(-1,1)^d.$   
Let $L^2=L^2((-1,1)^d,\mu),$ and denote by $\{\tilde{P}_k^{\alpha,\beta}\}_{k\in \N^d}$ the system of $L^2$ normalized Jacobi polynomials, see \cite[Section 7.1]{NS2} and \cite[Chapter 4]{Szeg1}. These Jacobi polynomials are our functions $\fk$ in this section, namely
$$\fk=\tilde{P}_{k}^{\alpha,\beta},\qquad k\in \N.$$ 
The Jacobi operator is symmetric on  $C_c^{\infty}((-1,1)^d)$ and has a self-adjoint extension 
$$Lf=\sum_{k\in \N^d}\la_k\sk{f}{\tilde{P}_{k}^{\alpha,\beta}}_{L^2}\tilde{P}_{k}^{\alpha,\beta},$$
where $\la_k=\sum_{i=1}^d\la_{k_i}^i$ with $\la_{k_i}^i=k_i(k_i+\alpha_i+\beta_i+1),$ $i=1,\ldots,d.$ Similarly to the previous two sections the projection $\Pi$ is
$$\Pi f=\sum_{k\in \N^d,k\neq 0}\sk{f}{\tilde{P}_{k}^{\alpha,\beta}}_{L^2}\tilde{P}_{k}^{\alpha,\beta},\qquad f\in L^2.$$
Moreover, $\Pi f=f$ precisely when $\int_{X}f(y)\,d\mu(y)=0$ and we have
\begin{equation}
\label{eq:PiJaPol}
\|\Pi f\|_p \leq 2\|f\|_p,\qquad 1\leq p\leq \infty.
\end{equation}   

The action of $\delta_i$ on Jacobi polynomials is given by 
\begin{equation} \label{eq:derJacPol}\delta_i\tilde{P}_{k}^{\alpha,\beta}=\sqrt{k_i(k_i+\alpha_i+\beta_i+1)}\sqrt{1-x_i^2}\tilde{P}_{k-e_i}^{\alpha+e_i,\beta+e_i},\end{equation}
and the Riesz transform becomes
$$R_i f=\sum_{k\in \N^d,k\neq 0}\bigg(\frac{\la_{k_i}^i}{\la_k}\bigg)^{1/2}\sk{f}{\tilde{P}_{k}^{\alpha,\beta}}_{L^2}\sqrt{1-x_i^2}\,\tilde{P}_{k-e_i}^{\alpha+e_i,\beta+e_i},\qquad f\in L^2.$$  
Dimension and parameter free estimates for single Riesz transforms $R_i$ are due to Nowak and Sj\"ogren \cite{NSj2}, who proved them for $\alpha,\beta\in[-1/2,\infty)^d.$

An application of Theorem \ref{thm:main} generalizes \cite[Theorem 5.1]{NSj2} to the vectorial Riesz transforms 
${\bf R}f=(R_1 f,\ldots, R_d f).$ This result is new according to our knowledge.
Moreover, we obtain an explicit estimate which is linear in $p^*.$
\begin{thm}
	\label{thm:JaPol}
	Fix $\alpha,\beta\in [-1/2,\infty)^d$ and $1<p<\infty.$ Then, for $f\in L^p$ which satisfy $\int_{X}f(y)\,d\mu(y)=0,$  we have 
	\begin{equation}
	\label{eq:JaPol}
	\left\|{\bf R} f\right\|_p\leq 24\,(p^*-1)\|f\|_{p},\qquad f\in L^p.
	\end{equation}
\end{thm}
\begin{remark}
	As in the previous two sections \eqref{eq:JaPol} holds for all $f\in L^p$ with $48\,(p^*-1)$ in place of $24\,(p^*-1).$ This follows from \eqref{eq:PiJaPol}.
\end{remark}
\begin{proof}
We are going to apply Theorem \ref{thm:main}, so we need to verify its assumptions for parameters $\alpha,\beta\in[-1/2,\infty)^d.$
	
	By \eqref{eq:quanJaPol} we see that if $\alpha,\beta\in [-1/2,\infty)^d,$ then  \eqref{eq:A1} and \eqref{eq:A2} (with $K=0$) are satisfied. Similarly, using \eqref{eq:derJacPol} one can see that, for such $\alpha$ and $\beta,$ the  conditions \eqref{eq:T1} and \eqref{eq:T2} also hold. The assumption \eqref{eq:T3} follows from Lemma \ref{lem:BC1} together with \eqref{eq:derJacPol}.
	
Now, since $\int_{X} f(y)\, d\mu(y)=0$ implies  $\Pi f=f,$ an application of Theorem \ref{thm:main} completes the proof of Theorem \ref{thm:JaPol}.
\end{proof}

\subsection{Harmonic oscillator - Hermite function expansions}
\label{HarmOsc}
Here we take
\begin{align*}p_i&=1,\quad q_i=x_i ,\quad a_i=1,\quad w_i(x_i)=1,\quad d\mu_i(x_i)=dx_i,\quad X_i=\R,\end{align*}
so that
\begin{equation}
\label{eq:quanHarmOsc}
\delta_i=\partial_{x_i}+x_i,\quad \d_i=\partial_{x_i},\quad  \delta_i^*=-\partial_{x_i}+x_i,\quad 
v_i=[\delta_i,\delta_i^*]=2,\qquad r(x)=|x|^2,
\end{equation}
and $L$ is the harmonic oscillator
$$
L
=\sum_{i=1}^{d}L_i
=-\Delta+|x|^2.
$$
It is well known that $L$ is essentially self-adjoint on $C_c^{\infty}(\R^d)$ with the self-adjoint extension given by
$$Lf=\sum_{k\in \N^d}(2|k|+d)\sk{f}{h_{k}}_{L^2}h_{k};$$
here $L^2=L^2(\R^d,dx),$ while $\{h_k\}_{k\in \N^d}$ is the system of $L^2$ normalized Hermite functions, see \cite[Section 7.4]{NS2}. The functions $h_k$ are our $\fk$'s in this section. They are of the form $h_k=h_{k_1}\otimes\cdots\otimes h_{k_d}$ where \begin{equation}
\label{eq:Hermfunform}
h_{k_i}(x_i)=\tilde{H}_{k_i}(x_i)e^{-x_i^2/2},\qquad x_i\in \R,
\end{equation}
with $\tilde{H}_{k_i}$ being the Hermite polynomial from Section \ref{OUop}. Note that as $0$ is not an $L^2$ eigenvalue of $L$  the projection $\Pi$ equals the identity operator.   

Next
\begin{equation}
\label{eq:donHermFunc}
\delta_i h_k=\sqrt{2k_j}h_{k-e_j},\end{equation}
and thus the Riesz transform is
$$R_i f=\sum_{k\in \N^d,k\neq 0}\bigg(\frac{2k_j}{2|k|+d}\bigg)^{1/2}\sk{f}{h_{k}}_{L^2}h_{k-e_i},\qquad f\in L^2.$$  
Here dimension-free bounds for the vector of Riesz transforms can be deduced, by means of transference, from the paper of Coulhon, M\"uller, and Zienkiewicz \cite{CMZ} (see also \cite{HRST} and \cite{LP2} for different proofs). Moreover, a dimension-free bound for the vector of Riesz transforms which is additionally linear in $p^*$  was proved by Dragi\v{c}evi\'c and Volberg in \cite[Proposition 4]{DV-Sch}.

Using Theorem \ref{thm:main} we are able to obtain a more explicit estimate for the vector ${\bf R}f$ than in \cite{DV-Sch}. However, contrary to \cite{DV-Sch}, our method says nothing about the vector of 'adjoint' transforms ${\bf R}^*f=(\delta_1^*L^{-1/2}f,\ldots, \delta_d^* L^{-1/2}f).$ 
\begin{thm}
	\label{thm:HarmOsc}
	For $1<p<\infty$ we have 
	\begin{equation*}
	\left\|{\bf R} f\right\|_p\leq 48(p^*-1)\|f\|_{p},\qquad f\in L^p.
	\end{equation*}
\end{thm}
\begin{proof}
	We apply Theorem \ref{thm:main}. In order to do so we need to check that its assumptions are satisfied.
	
	The equation \eqref{eq:quanHarmOsc} gives \eqref{eq:A1} and \eqref{eq:A2} with $K=1$. Condition \eqref{eq:T1} is straightforward. The assumption \eqref{eq:T2} holds since, by \eqref{eq:Hermfunform}, Hermite functions $h_{k_i}$ vanish rapidly at $\pm \infty$. Finally, \eqref{eq:T3} follows from \eqref{eq:donHermFunc} and the (well-known) density of Hermite functions in $L^p,$ $1\leq p<\infty.$
	
	Thus, an application of Theorem \ref{thm:main} is justified and the proof of Theorem \ref{thm:HarmOsc} is completed.
\end{proof}

\subsection{Laguerre operator - Laguerre function expansions of Hermite type}

\label{LagHerm}
For a parameter $\alpha\in (-1,\infty)^d$ we consider
\begin{align*}p_i&=1,\quad q_i=x_i-\frac{\alpha_i+1/2}{x_i} ,\quad a_i=1,\quad w_i(x_i)=1,\quad d\mu_i(x_i)=dx_i,\quad X_i=(0,\infty),\end{align*}
so that
\begin{equation}
\label{eq:quanLagHerm}
\begin{split}
\delta_i&=\partial_{x_i}+x_i-\frac{\alpha_i+1/2}{x_i},\quad \d_i=\partial_{x_i},\quad \delta_i^*=-\partial_{x_i}+x_i-\frac{\alpha_i+1/2}{x_i},\\
v_i&=[\delta_i,\delta_i^*]=2,\qquad r(x)=|x|^2+\sum_{i=1}^d\frac{\alpha_i^2-1/4}{x_i^2}.
\end{split}
\end{equation}
Here $L$ is the Laguerre operator 
$$
L
=\sum_{i=1}^{d}L_i
=-\Delta+|x|^2+\sum_{i=1}^d\frac{\alpha_i^2-1/4}{x_i^2}.
$$
Then $L$ is symmetric on $C_c^{\infty}(\R^d)$ and has a self-adjoint extension given by
$$Lf=\sum_{k\in \N^d}(4|k|+2d+2|\alpha|)\sk{f}{\fk^{\alpha}}_{L^2}\fk^{\alpha}.$$
In the above formula we denote $|k|=k_1+\cdots +k_d$ and $|\alpha|=\alpha_1+\cdots +\alpha_d;$ note that $|\alpha|$ may be negative. By $L^2$ we mean $L^2((0,\infty)^d,dx),$ while $\{\fk^{\alpha}\}_{k\in \N^d}$ stands for the system of $L^2$ normalized Laguerre functions of Hermite type, see \cite[Section 7.5]{NS2}. The functions $\fk^{\alpha}$ are the tensor products $\fk^{\alpha}=\varphi_{k_1}^{\alpha_1}\otimes\cdots\otimes \varphi_{k_d}^{\alpha_d}$ 
with
\begin{equation}
\label{eq:LagHermfuncform}
\varphi_{k_i}^{\alpha_i}(x_i)=\sqrt{2}\,\tilde{L}_{k_i}^{\alpha_i}(x_i^2)\,x_i^{\alpha_i+1/2}e^{-x_i^2/2},\qquad x_i>0,
\end{equation}
and $\tilde{L}_{k_i}^{\alpha_i}$ being the Laguerre polynomials from Section \ref{LagPol}.
In this section we take $$\fk=\fk^{\alpha}.$$ As $0$ is not an $L^2$ eigenvalue of $L$  the projection $\Pi$ equals the identity operator.

Next
\begin{equation}
\label{eq:donLagHerm}
\delta_i \fk^{\alpha}=-2\sqrt{k_j}\f_{k-e_j}^{\alpha+e_j},\end{equation}
and thus the Riesz transform is
$$R_i f=-\sum_{k\in \N^d,k\neq 0}\bigg(\frac{4k_i}{4|k|+2|\alpha|+2d}\bigg)^{1/2}\sk{f}{\fk^{\alpha}}_{L^2}\f_{k-e_j}^{\alpha+e_j},\qquad f\in L^2.$$ 
Dimension-free bounds for single Riesz transforms $R_i$ were obtained by Stempak and the author \cite[Theorem 5.1]{StWr1} for a certain restricted range of the parameter $\alpha.$

In this section, for $\alpha\in (1/2,\infty)^d$  we denote \begin{equation*}
C(\alpha)=\max_{i=1,\ldots,d}\frac{\alpha_i+1/2}{\alpha_i-1/2}.\end{equation*} By using Theorem \ref{thm:main} we obtain the following strengthening of \cite[Theorem 5.1]{StWr1} in the case $\alpha\in (1/2,\infty)^d$.
\begin{thm}
	\label{thm:LagHerm}
	Let $\alpha\in (1/2,\infty)^d.$ Then, for $1<p<\infty,$ we have 
	\begin{equation*}
	\left\|{\bf R} f\right\|_p\leq 24(1+\sqrt{C(\alpha)})(p^*-1)\|f\|_{p},\qquad f\in L^p.
	\end{equation*}
\end{thm}
\begin{proof}
	We apply Theorem \ref{thm:main}. In order to do so we need to check that its assumptions are satisfied.
	
	The formula \eqref{eq:quanLagHerm} gives \eqref{eq:A1} and \eqref{eq:A2} for $\alpha\in (1/2,\infty)^d$ with $K=C(\alpha)$. Conditions \eqref{eq:T1} and \eqref{eq:T2} follow from \eqref{eq:LagHermfuncform} and \eqref{eq:donLagHerm}. Finally, \eqref{eq:T3} follows from \cite[Lemma 5.2]{No0} and \eqref{eq:donLagHerm}.
	
	Thus, an application of Theorem \ref{thm:main} is justified and the proof of Theorem \ref{thm:LagHerm} is completed.
\end{proof}

\subsection{Laguerre operator - Laguerre function expansions of convolution type}
\label{LagConv}
For a parameter $\alpha\in (-1,\infty)^d$ we consider
\begin{align*}p_i&=1,\quad q_i=x_i ,\quad a_i=2\alpha_i+2,\quad \\
w_i(x_i)&=x_i^{2\alpha_i+1},\quad d\mu_i(x_i)=w_i(x_i)dx_i,\quad\quad X_i=(0,\infty),\end{align*}
so that
\begin{equation}
\label{eq:quanLagConv}
\begin{split}
\delta_i&=\partial_{x_i}+x_i,\quad \d_i=\partial_{x_i},\quad \delta_i^*=-\partial_{x_i}+x_i-\frac{2\alpha_i+1}{x_i},\\
v_i&=[\delta_i,\delta_i^*]=2+\frac{2\alpha+1}{x_i^2},\qquad r(x)=|x|^2.
\end{split}
\end{equation}
Here $L$ is the Laguerre operator 
$$
L
=\sum_{i=1}^{d}L_i
=-\Delta+|x|^2-\sum_{i=1}^d\frac{2\alpha_i+1}{x_i}\partial_{x_i}.
$$
Then $L$ is symmetric on $C_c^{\infty}((0,\infty)^d)$ and has a self-adjoint extension given by
$$Lf=\sum_{k\in \N^d}(4|k|+2d+2|\alpha|)\sk{f}{\ell_k^{\alpha}}_{L^2}\ell_k^{\alpha};$$
here $L^2=L^2((0,\infty)^d,w(x)dx),$ while $\{\ell_k^{\alpha}\}_{k\in \N^d}$ is the system of $L^2$ normalized Laguerre functions of convolution type, see \cite[Section 7.6]{NS2}. The functions $\ell_k^{\alpha}$ are of the form $\ell_k^{\alpha}=\ell_{k_1}^{\alpha_1}\otimes\cdots\otimes \ell_{k_d}^{\alpha_d}$ 
with
\begin{equation}
\label{eq:LagConvfuncform}
\ell_{k_i}^{\alpha_i}(x_i)=\sqrt{2}\,\tilde{L}_{k_i}^{\alpha_i}(x_i^2)\,e^{-x_i^2/2},\qquad x_i>0,
\end{equation}
and $\tilde{L}_{k_i}^{\alpha_i}$ being the Laguerre polynomials from Section \ref{LagPol}.
In this section we take $$\fk=\ell_k^{\alpha}.$$ Also here, as $0$ is not an $L^2$ eigenvalue of $L,$ the projection $\Pi$ equals the identity operator.

Next
\begin{equation}
\label{eq:donLagConv}
\delta_i \ell_k^{\alpha}=-2\sqrt{k_i}\,x_i\,\ell_{k-e_i}^{\alpha+e_i},\end{equation}
and thus the Riesz transform is
$$R_i f=-\sum_{k\in \N^d,k\neq 0}\bigg(\frac{4k_i}{4|k|+2|\alpha|+2d}\bigg)^{1/2}\sk{f}{\ell_k^{\alpha}}_{L^2}\ell_{k-e_i}^{\alpha+e_i},\qquad f\in L^2.$$  
The boundedness of these Riesz transforms on $L^p$ was proved by Nowak and Stempak, see \cite[Theorem 3.4]{NS3}. Later Nowak and Szarek \cite[Theorem 4.1]{NSZ1} enlarged the range of admitted parameters $\alpha.$ In both of these papers the Calder\'on-Zygmund theory was used, thus the $L^p$ bounds depended on the dimension $d.$ Applying Theorem \ref{thm:main} we obtain a dimension-free bound for the vectorial Riesz transform ${\bf R} f=(R_1 f,\ldots, R_d f).$  
\begin{thm}
	\label{thm:LagConv}
	Let $\alpha\in [-1/2,\infty)^d.$ Then, for $1<p<\infty,$ we have 
	\begin{equation}
	\label{eq:LagConv}
	\left\|{\bf R} f\right\|_p\leq 48(p^*-1)\|f\|_{p},\qquad f\in L^p.
	\end{equation}
\end{thm}
\begin{proof}
A continuity argument based on \eqref{eq:LagConvfuncform} and \eqref{eq:donLagConv} shows that it suffices to prove \eqref{eq:LagConv} for $\alpha\in(-1/2,\infty)^d.$
	We are going to apply Theorem \ref{thm:main}. In order to do so we need to check that its assumptions are satisfied.
	
	The formula \eqref{eq:quanLagConv} gives \eqref{eq:A1} and \eqref{eq:A2} with $K=1$. Conditions \eqref{eq:T1} and \eqref{eq:T2} follow from \eqref{eq:LagConvfuncform} and \eqref{eq:donLagConv}. It remains to prove \eqref{eq:T3}. For the space $\mD$ this condition follows from \cite[Lemma 4.3]{No0}. In the case of $\mD_i,$ $i=1,\ldots,d,$ the assumption \eqref{eq:T3} can be deduced from \eqref{eq:T3} for $\mD$ together with \eqref{eq:donLagConv}.   
	
	Thus, an application of Theorem \ref{thm:main} is justified and the proof of Theorem \ref{thm:LagConv} is completed.
\end{proof}

\subsection{Jacobi operator - Jacobi function expansions}
\label{JacFun}
For parameters $\alpha,\beta\in (-1,\infty)^d$ we consider
\begin{align*}p_i&=1,\quad q_i=-\frac{2\alpha_i+1}{4}\cot\frac{x_i}{2}+\frac{2\beta_i+1}{4}\tan\frac{x_i}{2} ,\quad a_i=(\alpha_i+\beta_1+1)^2/4,\\ w_i(x_i)&=1,\quad d\mu_i(x_i)=dx_i,\quad X_i=(0,\pi),\end{align*}
so that
\begin{equation}
\label{eq:quanJacFun}
\begin{split}
\delta_i&=\partial_{x_i}-\frac{2\alpha_i+1}{4}\cot\frac{x_i}{2}+\frac{2\beta_i+1}{4}\tan\frac{x_i}{2},\quad \d_i=\partial_{x_i},\\ \delta_i^*&=-\partial_{x_i}-\frac{2\alpha_i+1}{4}\cot\frac{x_i}{2}+\frac{2\beta_i+1}{4}\tan\frac{x_i}{2},\\
v_i&=[\delta_i,\delta_i^*]=\frac{2\alpha_i+1}{8\cos^2\frac{x_i}{2}}+\frac{2\beta_i+1}{8\sin^2\frac{x_i}{2}},\\ r(x) &=\sum_{i=1}^d\frac{(2\alpha_i+1)^2}{16}\cot^2\frac{x_i}{2}+\frac{(2\beta_i+1)^2}{16}\tan^2\frac{x_i}{2}+\frac{(\alpha_i+\beta_i+1)^2-(2\alpha_i+1)(2\beta_i+1)}{16}.
\end{split}
\end{equation}
Here $L$ is the Jacobi operator
$$
L
=\sum_{i=1}^{d}L_i
=-\Delta+\sum_{i=1}^d\bigg(\frac{4\alpha_i^2-1}{16\sin^2\frac{x_i}{2}}+
	\frac{4\beta_i^2-1}{16\cos^2\frac{x_i}{2}}\bigg).
	$$
Then $L$ is symmetric on $C_c^{\infty}((0,\pi)^d)$ and has a self-adjoint extension given by
$$Lf=\sum_{k\in \N^d}\la_k\sk{f}{\phi_k^{\alpha,\beta}}_{L^2}\phi_k^{\alpha,\beta};$$
here 
$\la_{k}=\sum_{i=1}^d\la_{k_i}^i$ with
$\la_{k_i}^i=(k_i+\frac{\alpha_i+\beta_i+1}{2})^2,$
$L^2=L^2((0,\pi)^d,dx),$ while $\{\phi_k^{\alpha,\beta}\}_{k\in \N^d}$ is the system of $L^2$ normalized Jacobi functions, see \cite[Section 7.7]{NS2}. These Jacobi function have the tensor product form $\phi_k^{\alpha,\beta}=\phi_{k_1}^{\alpha_1,\beta_1}\otimes\cdots\otimes \phi_{k_d}^{\alpha_d,\beta_d}$ with
\begin{equation}
\label{eq:JacFuncForm}
\phi_{k_i}^{\alpha_i,\beta_i}(x_i)=2^{(\alpha_i+\beta_i+1)/2}\tilde{P}_{k_i}^{\alpha_i,\beta_i}(\cos x_i)\,\bigg(\sin \frac{x_i}{2}\bigg)^{\alpha_i+1/2}\,\bigg(\cos \frac{x_i}{2}\bigg)^{\beta_i+1/2},
\end{equation}
for $x_i\in (0,\pi),$
and $\tilde{P}_{k_i}^{\alpha_i,\beta_i}$ being the Jacobi polynomials from Section \ref{sec:Jac}. In this section we take $$\fk=\phi_k^{\alpha,\beta}.$$ In the case when $\alpha,\beta\in [1/2,\infty)^d$ the $L^2$ kernel of $L$ is trivial, and thus the projection $\Pi$ equals the identity operator. 


Next
\begin{equation}
\label{eq:donJacFun}
\delta_i \phi_k^{\alpha,\beta}=-\sqrt{k_i(k_i+\alpha_i+\beta_i+1)}\phi_{k-e_i}^{\alpha+e_i,\beta+e_i},\end{equation}
and thus the Riesz transform is
$$R_i f=-\sum_{k\in \N^d,k_i\neq 0}\bigg(\frac{k_i(k_i+\alpha_i+\beta_i+1)}{\la_k}\bigg)^{1/2}\sk{f}{\phi_k^{\alpha,\beta}}_{L^2}\phi_{k-e_i}^{\alpha+e_i,\beta+e_i},\qquad f\in L^2.$$ 
In the case $d=1$ the $L^p$ boundedness of these Riesz transforms was proved by Stempak in \cite{Ste1}. Using Theorem \ref{thm:main} we obtain the following multi-dimensional bounds.

\begin{thm}
	\label{thm:JacFun}
	Let $\alpha,\beta\in [1/2,\infty)^d.$ Then, for $1<p<\infty,$ we have 
	\begin{equation*}
	\left\|{\bf R} f\right\|_p\leq 48\,(p^*-1)\|f\|_{p},\qquad f\in L^p.
	\end{equation*}
\end{thm}
\begin{proof}
	A continuity argument 
	based on \eqref{eq:JacFuncForm} and \eqref{eq:donJacFun} allows us to focus on $\alpha,\beta\in (1/2,\infty)^d.$ We are going to apply Theorem \ref{thm:main} for such parameters $\alpha$ and $\beta$. In order to do so we need to check that its assumptions are satisfied.
	
	The formula \eqref{eq:quanJacFun} gives \eqref{eq:A1} and \eqref{eq:A2} (with $K=1$). Conditions \eqref{eq:T1} and \eqref{eq:T2} follow from \eqref{eq:JacFuncForm} and \eqref{eq:donJacFun},  while \eqref{eq:T3} can be deduced from the density of polynomials in $C((-1,1))$ together with \eqref{eq:JacFuncForm} and \eqref{eq:donJacFun}.
	
	Thus, an application of Theorem \ref{thm:main} is permitted and the proof of Theorem \ref{thm:JacFun} is completed.
\end{proof}

\section*{Acknowledgments} This paper grew out of discussions with Oliver Dragi\v{c}evi\'c during the author's visit at the University of Ljubljana in November 2014. The author is greatly indebted to Oliver Dragi\v{c}evi\'c for these discussions and correspondence on the subject of the article. The author is also very grateful to Adam Nowak for clarifications on \cite{NS2} and many helpful remarks, and to the referee for his very careful reading of the manuscript and many valuable remarks.

Part of the research presented in this paper was carried over while the author was 'Assegnista di ricerca'
at the Universit\`a di Milano-Bicocca, working under the mentorship of Stefano Meda. The research was supported by Italian PRIN 2010 ``Real and complex manifolds:
geometry, topology and harmonic analysis"; Polish funds for sciences, National Science Centre (NCN), Poland, Research Project 2014\slash 15\slash D\slash ST1\slash 00405; and by the Foundation for Polish Science START Scholarship.

\end{document}